\documentclass[a4paper,11pt,reqno]{amsart}
\usepackage[utf8]{inputenc}
\usepackage[T1]{fontenc}

\usepackage{paralist}
\usepackage{letterspace}
\usepackage{booktabs}
\usepackage{ifdraft}
\usepackage{cite}
\usepackage{microtype}
\usepackage{tikz,tikz-cd}
\usepackage{lmodern}
\ifdraft{\usepackage[notcite,notref]{showkeys}}{}

\usepackage{todonotes}

\usepackage{amsmath,amssymb,amsthm}
\usepackage{mathtools,mathrsfs}
\usepackage{dsfont}
\usepackage{stmaryrd}
\usepackage{geometry}

\usepackage[hidelinks]{hyperref}

\newcommand{\IN}{\mathbb{N}}
\newcommand{\IR}{\mathbb{R}}
\newcommand{\Z}{\mathbb{Z}}
\newcommand{\bP}{\mathbf{P}}
\newcommand{\bE}{\mathbf{E}}
\newcommand{\bVar}{\mathbf{Var}}
\newcommand{\bCov}{\mathbf{Cov}}
\newcommand{\kB}{\mathfrak{B}}

\newcommand{\one}{\mathds{1}}

\newcommand{\Br}{\mathsf{Br}}
\newcommand{\Slr}{\mathsf{Sl}^+}
\newcommand{\Soo}{\mathsf{S}_\infty}
\newcommand{\Cl}{\mathsf{Cl}}
\newcommand{\Geo}{\mathsf{Geo}}
\newcommand{\Ber}{\mathsf{Ber}}
\newcommand{\shape}{\mathscr{S}}
\newcommand{\F}{\mathscr{F}}

\newcommand{\fdf}[1]{(\!(#1)\!)}

\newcommand{\mb}{\mathbf}

\newcommand{\ssq}{\subseteq}
\newcommand{\str}{\operatorname{str}}

\newtheorem{lemma}{Lemma}
\newtheorem{theorem}[lemma]{Theorem}
\theoremstyle{definition}
\newtheorem{definition}[lemma]{Definition}
\newtheorem{remark}[lemma]{Remark}

\title{Concatenating random matchings}
\author[F. Burghart]{Fabian Burghart}
\address{Department of Mathematics and Computer Science, Eindhoven University of Technology, 5612AE Eindhoven, The Netherlands}
\email{f.burghart@tue.nl}
\author[P. Th\'{e}venin]{Paul Th\'{e}venin}
\address{LAREMA, University of Angers, 49000 Angers, France}
\email{paul.thevenin@univie.ac.at}
\date{29 October 2024}

\begin{document}

\begin{abstract}
We consider the concatenation of $t$ uniformly random perfect matchings on $2n$ vertices, where the operation of concatenation is inspired by the multiplication of generators of the Brauer algebra $\mathfrak{B}_n(\delta)$. For the resulting random string diagram $\mathsf{Br}_n(t)$, we observe a giant component if and only if $n$ is odd, and as $t\to\infty$ we obtain asymptotic results concerning the number of loops, the size of the giant component, and the number of loops of a given shape. Moreover, we give a local description of the giant component. These results mainly rely on the use of renewal theory and the coding of connected components of $\mathsf{Br}_n(t)$ by random vertex-exploration processes.
\end{abstract}

\maketitle



%
%

\section{Introduction and preliminaries}

\subsection*{Brauer algebras}

The Brauer algebra $\kB_n(\delta)$ is an associative algebra named after Richard Brauer who introduced it in his work \cite{Bra37} on the representation theory of the orthogonal group. It is a $\Z[\delta]$-algebra depending on a parameter $n$, where we can think of $\delta$ as an indeterminate. While it is possible to describe $\kB_n(\delta)$ in terms of generators and relations, the following diagrammatic approach is perhaps more accessible: enumerating $2n$ points by $X_1, \dots, X_n, Y_1, \dots, Y_n$, a basis of the algebra $\kB_n(\delta)$ is given by all perfect matchings of these points -- that is, partitions of $\{ X_1, \dots, X_n, Y_1, \dots, Y_n\}$ in $n$ sets of size two. In the paper, we will always draw $X_1, \dots, X_n$ vertically in this order, and $Y_1, \dots, Y_n$ on their right, see Figure~\ref{fig:small_brauer}, left. 

To explain the multiplication of two basis elements $A, B$, assume that $A$ is given as a matching on vertices $X_1, \dots, X_n, Y_1, \dots, Y_n$ as before, and that $B$ is given as a matching on vertices $Y_1, \dots, Y_n, Z_1, \dots, Z_n$. The product $A \cdot B$ can be interpreted in a nice graphical way: it is the matching on the vertices $X_1,\dots,X_n, Z_1,\dots,Z_n$ obtained by concatenation of $A$ and $B$, multiplied by a factor $\delta^r$ where $r$ is the number of closed loops  appearing in the concatenation, see Figure~\ref{fig:small_brauer}. The definition of the multiplication operation is then extended $\Z[\delta]$-bilinearly to all of $\kB_n(\delta)$. 

\begin{figure}
\centering
 \includegraphics[width=\textwidth]{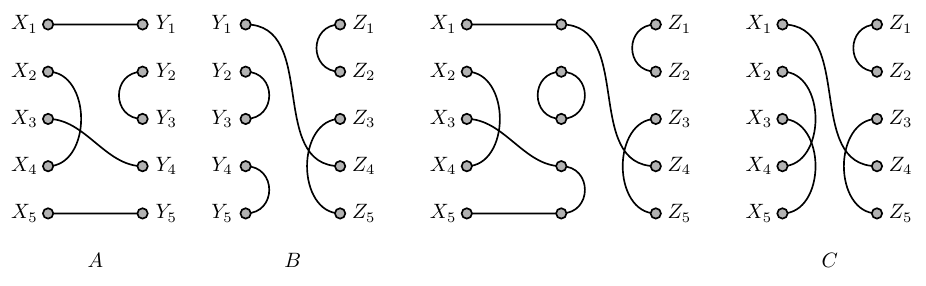}
 \caption{Two basis elements $A,B \in \kB_n(\delta)$. We have $A \cdot B = \delta C$.} 
 \label{fig:small_brauer}
\end{figure}

The Brauer algebra is one instance of a so-called diagram algebra, whose joint feature it is that the multiplication is defined through the concatenation of suitable diagrammatic generators. These diagram algebras have become a popular object of study in representation theory, algebraic topology, and theoretical physics. We briefly mention some other examples of diagram algebras: allowing only non-crossing matchings leads to the Temperley-Lieb algebra \cite{TL71}, whereas requiring that all edges have one endpoint in $\{X_1,\dots,X_n\}$ and the other in $\{Y_1,\dots,Y_n\}$ allows us to interpret matchings as permutations in the symmetric group $\mathfrak{S}_n$ ; in that case, since no closed loop is ever created by concatenating elements, the $\Z$-algebra of formal linear combinations of such diagrams is exactly the group ring $\Z\mathfrak{S}_n$. This means that one can regard the concatenation of matchings in the Brauer setting as a generalization of the multiplication of permutations, as was argued for instance in \cite{KM06}; the new feature gained by this extended generality is the non-trivial component structure in the diagrammatic representation, which we investigate in this paper. On the other hand, loosening the requirement of a matching to allowing any partition of $2n$ points leads to the partition algebras, see \cite{Ma94}. 

Introducing a slightly different spatial structure, where all strings have to go from left to right, but over- and under-crossings are distinguished, leads to braidings, and the multiplication of braidings by concatenation of diagrams up to isotopy defines a group structure, the so-called braid group \cite{Art25}. Random braidings specifically were considered in \cite{GGM13} and \cite{GT14}. 

By instead arranging the $2n$ points on the $x$-axis and concatenating two matchings whose edges are contained in the upper (resp. lower) half-plane without crossings, one obtains a meandric system \cite{FGG97}. The study of uniform random meandric systems has received some recent attention, see \cite{Kar20,FT22,JT23,BGP22}.

We emphasize that the primary interest of specifically considering unconditioned matchings is that their nice combinatorial structure allows for exact computations of all moments of the quantities if interest, as we will see.

\subsection*{Brauer diagrams}

Two positive integers $n,t$ being fixed, we call a graph obtained by the concatenation of $t$ perfect matchings of $2n$ vertices (keeping all vertices and edges) a \emph{Brauer diagram of length $t$}. We denote the $n(t+1)$ vertices of this graph by $(x_{i,j}, 0 \leq i \leq t, 1 \leq j \leq n)$ which we will assume to be arranged in a grid with $t+1$ columns of $n$ vertices each. For all $0 \leq i \leq t$, the set $\mathbf{X}_i := \{ x_{i,j}, 1 \leq j \leq n\}$ is called the $i$-th \emph{level} of the Brauer diagram. Within each level $i$, the vertices $x_{i,1},\dots,x_{i,n}$ are labelled from top to bottom (see Figure \ref{fig:Br57}). The matchings between neighbouring levels will also be called \emph{layers}, and we enumerate them from left to right by $1,\dots,t$. In other words, a Brauer diagram of length $t$ is given by $t$ layers between $t+1$ levels, where the $i$-th layer is a perfect matching on the vertex set $\mathbf{X}_{i-1} \cup \mathbf{X}_i$. See Figure~\ref{fig:Br57} for an example of a Brauer diagram of length $7$.

In the algebraic context of the Brauer algebra, Brauer diagrams of length $t$ form a basis for the $t$-fold tensor product $\kB_n(\delta)\otimes \dots\otimes \kB_n(\delta)$, where the tensor product is taken over $\Z[\delta]$. In particular, the differential maps in the bar complex (see e.g. \cite[Chapter IX.6]{CE56}) 
\[
 \dots \longrightarrow \kB_n(\delta)\otimes \kB_n(\delta) \otimes \kB_n(\delta) 
 \longrightarrow \kB_n(\delta) \otimes \kB_n(\delta)
 \longrightarrow \kB_n(\delta)
 \longrightarrow 0
\]
are given on basis elements $A_1,\dots,A_{t+1}$ by 
\[
 d(A_1\otimes \dots \otimes A_{t+1})=\sum_{i=1}^t (-1)^i A_1\otimes \dots \otimes A_i\cdot A_{i+1} \otimes \dots \otimes A_{t+1}
\]
where the summands on the right-hand side can be interpreted as arising from Brauer diagrams of length $t+1$ by evaluating the product of two consecutive layers. This construction can be used for the explicit computation of the homology of the Brauer algebras, see \cite{BHP21} -- however, due to the necessity of evaluating the chain maps on $(2n-1)!!^t$ basis elements, these computations quickly become unfeasible. Understanding the component structure of long typical Brauer diagrams is therefore one of the key motivations for this paper, in which we will investigate Brauer diagrams obtained from concatenating random matchings..

\subsection*{Random Brauer diagrams}

For $n, t \geq 1$, consider a uniform random Brauer diagram of length $t$. We are mainly interested in the asymptotic behaviour of this random variable at $n$ fixed, as the number $t$ of matchings grows. Observe that, by definition, this uniform Brauer diagram is obtained by concatenating $t$ i.i.d. matchings on $2n$ vertices (or, alternatively, it is a uniformly chosen basis element of $\kB_n(\delta)^{\otimes t}$). Hence, it is natural to consider the coupling process $\Br_n := (\Br_n(s), s \geq 1)$ constructed as follows: let $(\Pi_u, u \geq 1)$ be a family of i.i.d. uniform matchings on $2n$ points. For all $s \geq 1$, we define $\mathsf{Br}_n(s)$ as the Brauer diagram obtained by concatenation of $\{ \Pi_u, 1 \leq u \leq s \}$. It is clear by definition that, for all $s \geq 1$, $\mathsf{Br}_n(s)$ is a uniform Brauer diagram of size $t$. Unless specified, we will always work under this natural coupling. See Figure~\ref{fig:Br57} for a sampling of $\Br_5(7)$. 


\begin{figure}[h]
 \centering
 \includegraphics[width=0.9\textwidth]{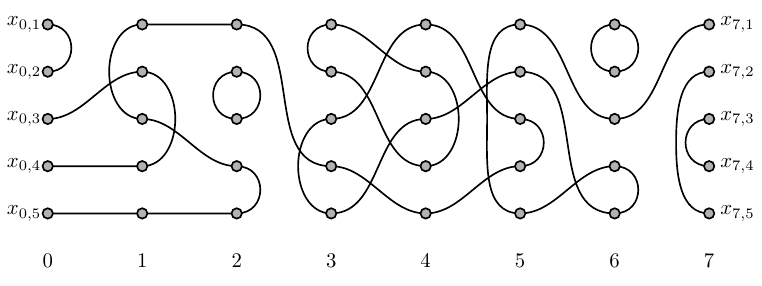}
 \caption{A possible outcome for $\Br_5(7)$, where the levels are indicated by the numbers underneath.}\label{fig:Br57}
\end{figure}

In the paper, the integer $n$ will always be fixed. Our goal is to understand the asymptotic structure of $\mathsf{Br}_n(t)$, as $t \rightarrow \infty$. In particular, we are interested in the connected components of this graph, which are of three different types: 
\begin{itemize}
 \item \emph{Closed loops.} These components form cycles on an even number of vertices. They cannot contain any vertex in levels $0$ or $t$. Note that it is possible to have a closed loop on only 2 vertices. 
 \item \emph{Transverse strings.} These components are paths with one endpoint in level 0 and the other in level $t$. 
 \item \emph{Slings.} These components are paths with either both endpoints in level 0 or both in level $t$.
\end{itemize}
It is easy to see that every closed loop or sling contains an even number of vertices of each level. The Brauer diagram in Figure~\ref{fig:Br57} has a unique transverse string, three closed loops and four slings.

For closed loops we further distinguish between different shapes, in the following ways:

\begin{definition}\label{def:shape}
 Let $K$ be a closed loop on $2\ell$ vertices in $\Br_n(t)$, and let $s$ be the leftmost level that contains vertices in $K$.
 \begin{enumerate}[(i)]
  \item We say that $K$ is of \emph{weak shape} $\mb a := (a_k)_{k\geq 0}$ if for all $k\geq0$, the number of vertices in level $s+k$ is given by $2a_k$.
  \item Define an exploration of $K$ by starting from the topmost vertex of $K$ in level $s$, and following along the edges starting towards the right. Writing down whether the next edge ``goes across'' (denoted by $A$), i.e. connects two vertices of different levels, or ``bends'' (denoted by $B$), i.e. connects two vertices in the same level, produces a word in $\{A,B\}^{\times 2\ell}$. Stop the exploration once it returns to the starting vertex. We say that $K$ is of \emph{strong shape $\shape\in\{A,B\}^{2\ell}$} if the exploration produces the word $\shape$. 
  \item We say that a strong shape $\shape$ is of weak shape $\mb a(\shape)$ if any closed loop $K$ with strong shape $\shape$ has weak shape $\mb a(\shape)$. The map $\mb a$ is well-defined, as will be shown in Lemma~\ref{lemma:automaton} below. 
  \item We define the \emph{stretch} of a sequence $\mb a$ to be $\str(\mb a)= \sup\{k\geq 0:a_k>0\}$. By a small abuse of language, we will talk about the stretch of a strong shape $\shape$ to denote the stretch of the associated weak shape $\mb a(\shape)$.
 \end{enumerate}
\end{definition}
Observe that a non-zero sequence $\mb a$ of non-negative integers can occur as the weak shape of a closed loop in $\Br_n(t)$ if and only if the following three conditions hold: (a) $a_k\leq \frac{n}{2}$ for all $k\geq 0$, (b) $a_k=0$ implies $a_{k+1}=0$ for all $k\geq 0$, and (c) $\str(\mb a) < \infty$.

\subsection*{Structure of the paper}
Before stating our main results in Section~\ref{sec:results}, we fix some notation below. Section~\ref{sec:RT} introduces the renewal theoretic setup for the proofs.  The main results concern properties of transverse strings, with proofs in Section~\ref{sec:proofsII}; limit laws for the total number of closed loops, with proofs in Section~\ref{sec:proofsI}; and limit laws for loops with given shape, with proofs in Section~\ref{sec:proofsIII}.

\subsection*{Notation}
Throughout the paper, we denote by $\IN$ and $\IN_0$ the positive and nonnegative integers, respectively. Let $n\in\IN$. We denote by $n!!$ the double factorial of $n$, that is,
\[
 n!! = \begin{cases} n(n-2)\cdots4\cdot2 & \text{ if $n$ is even}\\ n(n-2)\cdots3\cdot1 & \text{ if $n$ is odd.} \end{cases}
\]
Observe that $(2n-1)!!$ gives the number of distinct perfect matchings on $2n$ vertices. Since all matchings considered in this paper will be perfect matchings on $2n$ vertices, we will drop the attribute ``perfect'' for brevity. 

Furthermore, for two integers $n,k\geq 0$, we write
\[
 (n)_k=n(n-1)\cdots(n-k+1)
\]
for the $k$-th falling factorial of $n$. Note that $(n)_n=n!$, $(n)_k=0$ for $k>n$. For convenience, we set $(n)_0=1$. It will also be convenient to introduce the $k$-th falling double factorial of $n$, which we will denote by 
\[
 \fdf{n}_k=n(n-2)\cdots(n-2k+2).
\]

For any $n$, throughout the paper, we set $m:=\lfloor \frac{n}{2} \rfloor$, so that $n=2m$ if $n$ is even and $n=2m+1$ if $n$ is odd.

For a random variable $X$ and $1\leq r<\infty$, we denote the $L^r$-norm of $X$ by $\|X\|_r :=\bE[|X|^r]^{1/r}$. We denote the geometric distribution with success probability $p$ by $\Geo(p)$, that is, if $X \sim \Geo(p)$, then $\bP (X=k) = p (1-p)^{k-1}$ for all $k \geq 1$.

\subsection*{Main mathematical tools}

The model of random Brauer diagrams exhibits two main features which make its study very rich and interesting. 
First, the model has a very symmetric combinatorial structure. Therefore, many proofs throughout the paper are based on designing exploration processes, which are more involved variants of the following intuitive result. Fix $i \geq 1$, and explore the Brauer diagram starting from $x_{i-1,1}$ until hitting $\mathbf{X}_i$. Then, conditionally on $\Br_n(i-1)$, the level $\mathbf{X}_i$ is hit by the exploration at a uniform point.
This property allows us to compute exact probabilities of events such as a sling of $\mathbf{X}_{i-1}$ being part of a closed loop in $\Br_n(i)$. Computing probabilities of similar events then boils down to finding the right exploration process to consider. 
Second, the ``linear'' structure of long Brauer diagrams, together with certain recurring features which we call resets, allow for the use of renewal theory, which is an extremely powerful tool to prove central limit theorems and almost sure convergences, and could also be useful to investigate similar, more involved models inspired e.g. by random braids or non-crossing matchings. 
Unlike other similar models, the structure of Brauer diagrams allows us to obtain exact expressions of the quantities at stake, and not only their asymptotic behaviour.

\section{Main Results}
\label{sec:results}

In what follows, the integer $n$ is always fixed. Our goal is to understand the asymptotic structure of $\Br_n(t)$, as $t \to \infty$. 

We first investigate the behaviour of the transverse strings in $\Br_n(t)$. As we shall see in Lemma~\ref{lemma:tau} below, there is an almost surely finite stopping time $\tau$ with respect to the natural filtration of the process $\Br_n$ such that for all $t\geq\tau$ there is no transverse string in $\Br_n(t)$ if $n$ is even, and there is a unique transverse string in $\Br_n(t)$ if $n$ is odd. Therefore, we consider only the case where $n$ is odd in the next theorem. We obtain the convergence of the proportion of vertices in $\Br_n(t)$ belonging to the transverse string, and characterize the fluctuations of this proportion.

\begin{theorem}\label{theorem:Soo}
 Let $n=2m+1\geq 3$. If $t\geq\tau$, denote by $\Soo(t)$ the transverse string in $\Br_n(t)$, which is uniquely defined by definition of $\tau$, and by $|\Soo(t)|$ its number of vertices. For $t<\tau$, set $\Soo(t)=\emptyset$. 
 Then, as $t\to\infty$, we have 
 \begin{equation}\label{eq:SooLLN}
  \frac{|\Soo(t)|}{n(t+1)} \to \frac{2n+1}{3n},
 \end{equation}
 where convergence holds both almost surely and for the first moment. 

 Moreover, we have the following central limit theorem for $t\to\infty$:
 \begin{equation}\label{eq:SooCLT}
  \frac{|\Soo(t)|-\frac{2n+1}{3}t}{\sqrt{\frac{4}{135}(n-1)(n+2)(2n+1)t}} \to N(0,1)
 \end{equation}
 where convergence holds in distribution and for all moments. 
\end{theorem}

Defining $\Br_n(\infty)=\bigcup_{t\in\IN} \Br_n(t)$ to be the semi-infinite random Brauer diagram, we set $\Soo(\infty)$ to be the corresponding transverse string in $\Br_n(\infty)$ if $n$ is odd -- meaning that $\Soo(\infty)$ is the component of $\Br_n(\infty)$ that intersects every level. As before, existence and uniqueness of $\Soo(\infty)$ are guaranteed by Lemma~\ref{lemma:tau} below. We write $E(\Soo(\infty))$ for the edge set of $\Soo(\infty)$. We can even compute exact probabilities of local events, and thus characterise the local structure of $\Soo(\infty)$.

\begin{theorem}\label{theorem:Soolocal}
 Let $n=2m+1\geq 1$.
 \begin{enumerate}[(i)]
  \item Define $V_t=\left|\Soo(\infty)\cap \mathbf{X}_t\right|$ to be the number of vertices on $\Soo(\infty)$ at level $t$. Note that $V_t$ is necessarily odd. Then
  \begin{equation}\label{eq:Vdist}
   \bP[V_t=2j+1\mid  t \geq \tau]=\frac{\fdf{2m}_j}{\fdf{2m+1}_{j+1}}
  \end{equation}
  for $\ell=0,\dots,m$.
  \item Define $E_t = \left|\left\{e\in E(\Soo(\infty)): e\cap \mathbf{X}_t\neq \emptyset \neq e\cap \mathbf{X}_{t+1}\right\}\right|$ 
  to be the number of edges of $\Soo(\infty)$ in layer $t$ that connect levels $t$ and $t+1$. Note that $E_t$ is necessarily odd. Then
  \begin{equation}\label{eq:Edist}
   \bP[E_t=2\ell+1\mid  t \geq \tau]=\sum_{j=\ell}^m \sum_{k=\ell}^m \binom{j}{\ell}\binom{k}{\ell}\frac{\fdf{2m}_j \fdf{2m}_k}{\fdf{4m+1}_{j+k+1}}
  \end{equation}
  for $\ell=0,\dots,m$.
  \item Denote by $A_t$ the number of edges of $\Soo(t)$ that ``go across'' layers (in the sense of Definition~\ref{def:shape}) and by $B_t$ analogously the number of edges of $\Soo(t)$ that ``bend'', with the convention that $A_t=B_t=0$ if $t < \tau$. Then, as $t\to\infty$, we have 
  \begin{equation}\label{eq:ABLLN}
   \frac{1}{t}(A_t,B_t) \to \left(1+\frac{8m^2}{12m+3},\frac{8m^2+4m}{12m+3}\right)
  \end{equation}
  where convergence holds both almost surely and for the first moment.
 \end{enumerate}
\end{theorem}

Our next result concerns the number of connected components in $\Br_n(t)$. It turns out that this number again behaves linearly in $t$ and has Gaussian fluctuations. To ease notation, we will write 
\[
 \mu_n=\sum_{j=0}^{\lfloor n/2\rfloor-1} \frac{1}{2n-2j-1}=
       \begin{cases}
        \sum_{k=1}^m \frac{1}{n+2k-1} & \text{ if } n=2m\\
        \sum_{k=1}^m \frac{1}{n+2k} & \text{ if } n=2m+1
       \end{cases}
\]
and 
\[
 \sigma^2_n=\sum_{j=0}^{\lfloor n/2\rfloor-1} \frac{2n-2j-2}{(2n-2j-1)^2}=
            \begin{cases}
             \sum_{k=1}^m \frac{n+2(k-1)}{(n+2k-1)^2} & \text{ if } n=2m\\
             \sum_{k=1}^m \frac{n+2k-1}{(n+2k)^2} & \text{ if } n=2m+1
            \end{cases}
\]
where the second formulation is the one that will be obtained in the proof.

\begin{theorem}\label{theorem:CC}
 Denote by $C_n(t)$ the number of loops in $\Br_n(t)$. Then, as $t\to\infty$, 
 \begin{equation}\label{eq:CCLLN}
  \frac{C_n(t)}{t} \to \mu_n
 \end{equation}
 holds both almost surely and in for the first moment. Additionally, we have the following central limit theorem for $t\to\infty$:
 \begin{equation}\label{eq:CCCLT}
  \frac{C_n(t)-\mu_n t}{\sigma_n\sqrt{t}} \to N(0,1)
 \end{equation}
 where convergence occurs both in distribution and in all moments. 
\end{theorem}

\begin{remark}\label{rem:CC}
 Any connected component of $\Br_n(t)$ that is not a closed loop has two endpoints in $\mathbf{X}_{0} \cup \mathbf{X}_{t}$. Conversely, none of the vertices in these two levels can be part of a closed loop in $\Br_n(t)$. It follows that there are exactly $n$ components that are not loops and therefore that the results of Theorem~\ref{theorem:CC} hold (with the same constants $\mu_n$ and $\sigma_n$) if $C_n(t)$ is replaced by the number of connected components in $\Br_n(t)$. 
\end{remark}

%

Finally, our last result refines the previous one and allows us to compute the asymptotic number of loops in $\Br_n(t)$ with a given shape.

\begin{theorem}\label{theorem:shape}
 Let $n\geq 2$, and let $\shape$ be a strong shape for $\Br_n$ with weak shape $\mb a=(a_k)_{k\geq0}$. Set 
 \[
  \gamma_n^{\mb a} = \prod_{i\geq1} \binom{a_i+a_{i-1}-1}{a_i}
 \]
 and define $b_0=a_0$ and $b_i=a_i+a_{i-1}$ for all $i\geq 1$. Denote by $C_n^\shape(t)$  the number of closed loops in $\Br_n(t)$ with strong shape $\shape$, and by $C_n^{\mb a}(t)$ the number of closed loops with weak shape $\mb a$. Then, as $t\to\infty$,
 \begin{equation}\label{eq:sshapeLLN}
  \frac{C_n^\shape(t)}{t} \to \frac{1}{2a_0} \prod_{i\geq0} \frac{(n)_{2a_i}}{\fdf{2n-1}_{b_i}} =: \mu_n^\shape
 \end{equation}
 and
 \begin{equation}\label{eq:wshapeLLN}
  \frac{C_n^{\mb a}(t)}{t} \to \frac{\gamma_n^{\mb a}}{2a_0} \prod_{i\geq0} \frac{(n)_{2a_i}}{\fdf{2n-1}_{b_i}}
 \end{equation}
 both almost surely and for the first moment. Moreover, there exists a positive constant $\sigma_{n,\shape}^2$  such that 
 \begin{equation}\label{eq:sshapeCLT}
  \frac{C_n^\shape(t)-\mu_n^\shape t}{\sigma_{n,\shape}\sqrt{t}} \to N(0,1)
 \end{equation}
 and
 \begin{equation}\label{eq:wshapeCLT}
  \frac{C_n^{\mb a}(t)-\gamma_n^{\mb a}\mu_n^\shape t}{\sqrt{\left((\gamma_n^{\mb a})^2\sigma_{n,\shape}^2 - (\gamma_n^{\mb a})_2\mu_n^\shape\right)t}} \to N(0,1)
 \end{equation}
 as $t\to\infty$, where convergence holds in distribution and for all moments. 
\end{theorem}

\begin{remark}
 The constant $\sigma_{n,\shape}^2$ is explicit, and given by \eqref{eq:sigmashape}. 
\end{remark}


\begin{remark}
 In light of Theorems~\ref{theorem:Soo} and \ref{theorem:CC}, it might seem surprising that the expressions on the right-hand sides of \eqref{eq:sshapeLLN} and \eqref{eq:wshapeLLN} do not depend on the parity of $n$. Considering the cases $n=2$ and $n=3$ may serve as a sanity check. In both cases, the only possible strong shape for a closed loop on $2\ell$ vertices is $\shape_\ell:=A^{\ell-1}BA^{\ell-1}B$ where $\ell\geq 1$, and the corresponding weak shape is given by $a_k=1$ for $k\leq \ell$ and $a_k=0$ otherwise. For $n=2$ we obtain
 \begin{equation*}
  \lim_{t\to\infty} \bE\left[\sum_{\ell\geq1} 2\ell\frac{C_2^{\shape_\ell}(t)}{t} \right]
  = \sum_{\ell\geq1} 2\ell \lim_{t\to\infty} \bE\left[\frac{C_2^{\shape_\ell}(t)}{t} \right] 
  = \sum_{\ell\geq1} \frac{\ell 2^{\ell}}{3^{\ell+1}} = 2
 \end{equation*}
which is reasonable since the expected value on the left-hand side is the expected number of vertices per level that lie on closed loops. (We first swap summation and expectation, and then limit and summation, as all quantities involved are nonnegative). 
 For $n=3$, we obtain analogously
  \begin{equation*}
  \lim_{t\to\infty} \bE\left[\sum_{\ell\geq1} 2\ell\frac{C_3^{\shape_\ell}(t)}{t} \right]
  = \sum_{\ell\geq1} 2\ell \lim_{t\to\infty} \bE\left[\frac{C_3^{\shape_\ell}(t)}{t} \right] 
  = \sum_{\ell\geq1} \frac{\ell 6^\ell}{5^{\ell+1}3^{\ell-1}} = \frac{2}{3}.
 \end{equation*}
 Once again, this is reasonable since the expected number of vertices per level on the transverse string converges to $\frac{7}{3}$ by Theorem \ref{theorem:Soo}. 
\end{remark}

\begin{remark}
 The methods used in the proof of Theorem~\ref{theorem:shape} could in principle be used to also show limit theorems for the number of components of a given size $2\ell$. However, this approach leads to combinatorial complications which we were not able to overcome. 
\end{remark}

\begin{remark}[A remark on the proof strategies]\label{rem:modified}
 The proof of Theorem~\ref{theorem:CC} relies on limit theorems for sums of independent random variables, and the proofs for Theorems~\ref{theorem:Soo} and \ref{theorem:shape} rely on limit theorems for processes with regenerative increments, combined with a result on the uniform integrability of such processes, Lemma~\ref{lemma:SooUI}. The precise setup for renewal theory will be described in Section~\ref{sec:RT} below. 
 As consequence, it follows from a functional limit theorem like \cite[Theorem 5.47]{Serfozo09} that the processes $C_n(t), |\Soo(t)|, C_n^\shape(t)$ and $C_n^{\mb a}(t)$ converge in distribution after centering and rescaling to standard Brownian motions on $[0,1]$, in the usual Skorokhod space. Furthermore, by observing that linear combinations of processes with regenerative increments are again processes with regenerative increments, and using the Cram\'er-Wold device together with a lemma on uniform integrability (Lemma~\ref{lemma:SooUI}), one can show that the limit laws in \eqref{eq:sshapeCLT} and \eqref{eq:wshapeCLT} extend to multivariate limit laws for any finite collection of shapes. 
 
 To reduce technical obstacles in the proofs, we will use a \emph{modified Brauer diagram}, denoted $\Br'_n(t)$, where we draw an additional uniform random matching on $2m$ (uniformly chosen) vertices of level $0$ to the left of the usual $\Br_n(t)$. This has certain advantages: it eliminates slings with both endpoints in $\mb X_0$ and turns them into segments of closed loops or of a transverse string. More so, if $n$ is odd then the vertex in $\mb X_0$ that has no edge going out to its left is necessarily the endpoint of a transverse string, and the transverse string is unique in $\Br'_n(t)$ for all $t \geq 1$. The change from $\Br_n(t)$ to $\Br'_n(t)$ bears no consequences for the results presented above, as it changes the number of components or the length of the transverse string by at most a constant multiple of a geometrically distributed random variable. 
\end{remark}

\section{The sling process and renewal theory} 
\label{sec:RT}

Recall from the introduction that $\Br_n=\{\Br_n(t)\}_{t\in\IN}$ is a stochastic process, where $\Br_n(t+1)$ is obtained by concatenating a uniform random matching (independent from what happened in the prior $t$ layers) to the right of $\Br_n(t)$.  In particular, this process is Markovian with respect to its natural filtration.

Recall also from Remark~\ref{rem:modified} that $\Br'_n(t)$ denotes the modified random Brauer diagram, where $\Br'_n(0)$ consists of a maximum matching of the vertices in $\mb X_0$ drawn to their left. Clearly, the process $\Br'_n:=\{\Br_n(t)\}_{t\in\IN_0}$ is Markovian with respect to its natural filtration, which we will denote by $\F=\{\F_t\}_{t\in\IN_0}$. For notational convenience, we will write 
\[
 \bP_t[\ \cdot\ ] := \bP[\ \cdot\ | \F_t].
\] 


Denote now by $\Slr_n(t)$ the subgraph of $\Br'_n(t)$ consisting of all slings. Since we are considering the modified random Brauer diagram, the slings in $\Slr_n(t)$ necessarily have both endpoints in level $t$. Like $\Br'_n$, we will regard $\Slr_n:=\{\Slr_n(t)\}_{t\in\IN_0}$ as a stochastic process, which we will refer to as the \emph{sling process}. The sling process is also Markovian since $\Slr_n(t+1)$ only depends on $\Slr_n(t)$ (and the randomness of the layer $t+1$), but not on its history before time $t$. Throughout the paper, we will occasionally consider a distinguished sling $S$ in $\Slr_n(t)$. Since this sling by definition is a part of $\Br'_n(t)$, this entails that the choice of $S$ is independent of all layers to the right of level $t$.  

To complement $\Slr_n(t)$, we denote the subgraph of $\Br'_n(t)$ consisting of all closed loops by $\Cl_n(t)$.

With constant probability $p_0=\frac{(2m-1)!!^2}{(2n-1)!!}$ (if $n$ is even) and $p_0=\frac{n^2(2m-1)!!^2}{(2n-1)!!}$ (if $n$ is odd), layer $t$ contains no edges (in case $n$ is even) or exactly one edge (in case $n$ is odd) connecting level $t-1$ to level $t$. For example, this happens in layers $3$ and $7$ in Figure~\ref{fig:Br57}. As a consequence, almost surely, infinitely often the sling process $\Slr_n(t)$ contains $m$ slings. For $t$ such that $\Slr_n(t)$ has $m$ slings, we will call layer $t$ a \emph{reset}. In words, it resets the process $\Slr_n(t)$ in the sense that $\Slr_n(t)=\Slr_n(0)$ holds in distribution if and only if layer $t$ is a reset. 

The times $t$ at which resets occur can be regarded as \emph{renewal times} in the sense of renewal theory. We set $T_0=0$ and denote by $T_j$ for $j\geq 1$ the $j$-th reset time. Denote the length of the renewal intervals by $R_i:=T_i-T_{i-1}$ for $i\geq 1$. According to the preceding discussion, the sequence $R_1,R_2,\dots$ is a sequence of independent random variables with distribution $\Geo(p_0)$. By definition, the $j$-th reset time can also be written $T_j=\sum_{i=1}^j R_i$ for $j\geq 1$. Moreover, for $t\in\IN$, we denote by $N(t):=\max\{j\geq 0:T_j\leq t\}$ the renewal counting process associated with the renewal times.

A concept which we will refer to several times is that of a process \emph{with regenerative increments} over renewal times $T_j$: referring to \cite[Definition~2.52]{Serfozo09}, we say that an $\IR$-valued stochastic process $X=(X_t)_{t\geq0}$ has regenerative increments over $(T_j, j\geq 1)$, if the pairs 
\[
 \big(\xi_j,\{X(T_{j-1}+t) - X(T_{j-1})\}_{0\leq t\leq \xi_j}\big),
\]
indexed by $j$, are independent and identically distributed, where $\xi_j=T_j-T_{j-1}$ for $j\geq 1$. In other words, this means that $X$ has i.i.d. increments over each renewal interval. We point to \cite[Chapter~2]{Serfozo09} for a general introduction to renewal theory, and cite specific results as we need them. 

Note that for all $t\geq 0$, the graph $\Slr_n(t)$ consists of exactly $m$ connected components (i.e. slings), that are naturally subgraphs of $\Br'_n(t+1)$. Letting $S$ be any sling of $\Slr_n(t)$, three different scenarios are possible in $\Br'_n(t+1)$: firstly, $S$ might become part of a (larger) sling in $\Slr_n(t+1)$; secondly, $S$ might become part of a closed loop in $\Br'_n(t+1)$; and thirdly, $S$ might become part of a transverse string. The existence of resets implies that every sling in $\Slr_n(t)$ (for any value of $t$) will eventually become either a part of a transverse string or a part of a closed loop in $\Slr_n(t+\tau_S)$, for some $\tau_S$ stochastically bounded by $\Geo(p_0)$.

\section{The transverse string -- proofs}
\label{sec:proofsII}

We will first prove Theorem~\ref{theorem:Soo} about the asymptotic behaviour of the transverse string, when $n$ is odd. We begin by verifying the claims made prior to Theorem~\ref{theorem:Soo} about the existence and uniqueness of a transverse string. Note that for the purpose of this lemma, we return to the setting of the non-modified Brauer diagram $\Br_n(t)$. As such, the stopping times in the lemma below are adapted to the natural filtration of $\Br_n$. 

\begin{lemma}\label{lemma:tau}
For $n$ even, there exists a stopping time $\tau\leq T_1\sim\Geo(p_0)$ such that there is no transverse string in $\Br_n(t)$ for $t \geq \tau$. For $n=2m+1$, there exists a stopping time $\tau\leq T_1\sim\Geo(p_0)$ such that there is a unique transverse string in $\Br_n(t)$ for $t \geq \tau$. In both cases, $\Br_n(t)$ contains $2m$ slings for all $t\geq\tau$, exactly $m$ of which have both endpoints in level $t$. 
\end{lemma}

%

\begin{proof}
 Let us consider the number of slings in $\Br_n(t)$ with both endpoints in level $t$. By the earlier classification of the connected components of $\Br_n(t)$, the vertices in level $t$ are either the unique endpoints of transverse strings or endpoints of slings. The number of transverse strings is clearly non-increasing in $t$ because any transverse string in $\Br_n(t+1)$ uniquely restricts to a transverse string in $\Br_n(t)$. Furthermore, every sling has exactly two endpoints, so that the number of slings is non-decreasing in $t$. Additionally, for $n$ odd, for parity reasons at least one of the vertices in level $t$ belongs to a transverse string.
 
 On the other hand, the fact there are resets at times $T_1,T_2,\dots$ implies that for even $n$ there is no edge between levels $T_1-1$ and $T_1$, and for odd $n$ there is a unique edge between these levels. 
\end{proof}

\begin{remark}
Note that, in general, $\tau<T_1$, that is, the uniqueness or nonexistence of a transverse string occurs strictly before the first reset time.
\end{remark}


For $t \geq \tau$, we can consider an exploration of the transverse string starting from its endpoint in level 0. We then define $v_\infty(t)$ and $w_\infty(t)$ to be respectively the first and last vertices in level $t$ reached by this exploration. We will use this notation for both $\Br_n$ and $\Br_n'$. Alternatively, $v_\infty(t)$ is the unique endpoint of $\Soo(t)$ in $\Br_n(t)$ or $\Br'_n(t)$, and $w_\infty(t)$ is the starting point of the unique infinite component in $\Br_n(\infty)\setminus \Br_n(t)$. Observe that $w_\infty(t)$ is not adapted to the natural filtration of the diagram, while $v_\infty(t)$ is. 

For the proof of Theorem~\ref{theorem:Soo}, we work with the modified Brauer diagram $\Br'_n(t)$. For $n=2m+1$, this has the advantage that $\Slr_n(t)$ consists of $m$ slings for all $t\geq 1$. Thus the one vertex not matched up by $\Slr_n(0)$ is the starting point of a transverse string, and this transverse string is unique already at time $t=0$. This simplifies the following arguments that would otherwise require an extra condition such as $t\geq \tau$. We denote the transverse string in $\Br'_n(t)$ by $\Soo'(t)$.

We define $\Soo'(\infty)$ analogously to $\Soo(\infty)$ preceeding Theorem~\ref{theorem:Soolocal}. We also set 
\[
 \Soo''(t):=\{v\in\Br'_n(t): v\in\Soo'(\infty)\}, 
\]
that is, the set of all vertices up to level $t$ that are part of $\Soo'(t')$ for some $t'\geq t$ (and thus for all $t''\geq t'$) -- note that $\{v\in\Br'_n(t)\cap \Soo'(\infty)\}\in \F_{T_{N(t)+1}}$, since vertices up to level $t$ are either on the transverse string or in a closed loop in $\Br'_n(T_{N(t)+1})$. 

For $t \geq \tau$, observe that, as sets of vertices, $\Soo(t)\ssq \Soo'(t)\ssq \Soo''(t)$, and that $\Soo''(t)\setminus\Soo(t)$ consists of some of the slings in $\Br_n(t)$. 


In particular, 
\[
 |\Soo''(t)\setminus\Soo(t)|\leq nT_1 + n\big(t-T_{N(t)}\big)\leq n\big(R_1 + R_{N(t)+1}\big)
\]
so exchanging $\Soo(t)$ for $\Soo''(t)$ bears no consequence for Theorem~\ref{theorem:Soo}, since $t^{-1/2}n(R_1+R_{N(t)+1}) \to 0$ in probability, as $t\to\infty$. 

\subsection*{Law of large numbers} 
We first show convergence in expectation for the law of large numbers in \eqref{eq:SooLLN}, and then conclude almost sure convergence through the use of renewal theory. For $t<t'$ and $S$ a given sling in $\Slr_n(t)$, we denote by $S\ssq \Soo'(t')$ the event that $S$ becomes a part of the transverse string in $\Br'_n(t'+1)$ and by $S\ssq \Cl_n(t')$ the event that $S$ becomes a part of a closed loop in $\Br'_n(t'+1)$.

\begin{lemma}\label{lemma:SlFate}
Let $n=2m+1$ and denote by $S$ a distinguished sling in $\Slr_n(t)$. For $t'>t$, denote by $S\ssq \Soo'(t')$ the event that $S$ becomes a part of the transverse string in $\Br'_n(t'+1)$ and by $S\ssq \Cl_n(t')$ the event that $S$ becomes a part of a closed loop in $\Br'_n(t'+1)$. Then 
 \begin{equation}\label{eq:FateA}
  \bP_t[S\ssq\Soo'(t+1)]=\frac{2}{n+2}
 \end{equation}
 and 
 \begin{equation}\label{eq:FateB}
  \bP_t[S\ssq \Cl_n(t+1)]=\frac{1}{n+2}.
 \end{equation}
 Moreover, denote by $D(S)$ the smallest integer $d$ such that either one of the events $S\ssq\Soo'(t+d)$ or $S\ssq \Cl_n(t+d)$ has occurred. Then $D(S)\sim \Geo(\frac{3}{n+2})$. 
\end{lemma}

\begin{proof}
 We employ an exploration technique similar to the definition of a strong shape in Definition~\ref{def:shape}. The technique, which we will use several times in the paper, consists in exploring the component of a sling until one hits a given set of vertices. In $\Br'_n(t+1)$, follow along $\Soo'(t+1)$ starting from $v_\infty(t)$ towards the right. This exploration terminates once one of the $n$ vertices at level $t+1$ or one of the two endpoints of $S$ is reached. Thus $S\ssq \Soo'(t+1)$ if and only if the exploration reaches one of the two endpoints of $S$ before it reaches level $t+1$. Since the matching at layer $t+1$ is uniform and independent of $(S,v_\infty(t))$, this happens with probability $\frac{2}{n+2}$, proving \eqref{eq:FateA}. 
 
 For \eqref{eq:FateB}, start a similar exploration from the top-most endpoint of $S$ towards the right. Terminate this exploration if it either reaches a vertex in level $t+1$, or if it reaches $v_\infty(t)$, or if it reaches the second endpoint of $S$. Each of these $n+2$ vertices is reached first with equal probability, but only the third case corresponds to the event where $S$ is closed to a loop by layer $t+1$. Hence, $\bP_t[S\ssq \Cl_n(t+1)]$ is $\frac{1}{n+2}$.
 
In order to get the distribution of $D(S)$, observe first that the events $S\ssq\Soo'(t+1)$ and $S\ssq\Cl_n(t+1)$ are disjoint. On the other hand, if none occur, then $S$ is part of a (larger) sling in $\Slr_n(t+1)$. The claim follows, since layers are independent. 
\end{proof}

\begin{remark}\label{rem:X}
 Using the explorations of Lemma~\ref{lemma:SlFate}, it is possible to make a more refined statement about the number of slings in $\Slr_n(t)$ that become part of $\Soo'(t+1)$. If we denote this number by $X(t)$, then 
 \[
  \bP_t[X(t)=s]=2^s(2m+1)\frac{(m)_s}{\fdf{4m+1}_{s+1}}.
 \]
 Indeed, assume that $X(t)=s$. Then, exploring $\Soo'(t+1)$ starting from $\mathbf{X}_0$, the $s$ slings appear in one of $(m)_s$-many permutations. Moreover, every sling is traversed in one of two directions, contributing the factor $2^s$. After having traversed the $s$ slings, $\Soo'(t+1)$ then has to move on to one of $2m+1$ vertices on level $t+1$. Finally, one needs a specific set of $s+1$ edges to be present in the matching of layer $t+1$, and this happens with probability $\frac{(4m-1-2s)!!}{(4m+1)!!}=\frac{1}{\fdf{4m+1}_{s+1}}$. 
 
 By an analogous argument, if $X'(t)$ is the number of slings different from $S$ in $\Slr_n(t)$ that end up in the same closed loop in $\Br'_n(t+1)$ as the fixed sling $S$, then 
 \[
  \bP_t[S\ssq\Cl_n(t+1)\text{ and }X'(t)=s]=2^s \frac{(m-1)_s}{\fdf{4m+1}_{s+1}}.
 \]
 
\end{remark}


\begin{proof}[Proof of Theorem~\ref{theorem:Soo}, law of large numbers]
 Denote by $U_t$ a uniform vertex chosen independently of $\F_t$ among all $n(t+1)$ vertices in $\Br'_n(t)$, and denote its level by $L=L(U_t)$. Then 
 \begin{equation}
 \label{eq:exp}
  \frac{\bE[|\Soo''(t)|]}{n(t+1)}=\frac{1}{n(t+1)}\sum_v \bE[\one\{v\in\Soo''(t)\}]=\bP[U_t\in \Soo'(\infty)]
 \end{equation}
 for all $t$.  
 In analogy to Lemma~\ref{lemma:SlFate}, we denote by $D=D(U_t)$ the smallest integer $d$ such that $U_t$ is not any more part of a sling in $\Br'_n(L+d)$. Roughly speaking, $D$ is the (random) number of layers until the process $\Br'_n$ has decided whether $U_t$ is part of the transverse string or part of a closed loop. In particular, the random variables $D$ and $L$ are independent. We then obtain: 
 \begin{equation}\label{eq:pApB}
  \bP[U_t\in \Soo'(\infty)] = \sum_{d=0}^\infty \bP[D=d]\bP[U_t\in \Soo'(L+d)\mid D=d]
 \end{equation}
 The event $D=0$ means that $U_t$ is exactly the unique endpoint of $\Soo'(L)$, which happens with probability $\frac{1}{n}$. For $D>0$, the vertex $U_t$ is a uniformly chosen endpoint of one of the $m$ slings in $\Slr_n(L)$, and then whether or not $U_t$ ends up as a vertex on the transverse string is equivalent to whether or not its sling ends up as part of the transverse string or as part of a closed loop. In particular, by Lemma~\ref{lemma:SlFate}, $D$ is geometrically distributed when conditioned on $D>0$. Moreover, conditioned on $D=d>0$, the event $U_t\in\Soo'(L+d)$ happening means that $U_t$ lies on a sling $S$ in $\Slr_n(L+d-1)$ that gets incorporated into the transverse string in layer $L+d$, knowing that it otherwise becomes part of $\Cl_n(L+d)$. The former happens with probability $\frac{2}{3}$ according to Lemma~\ref{lemma:SlFate}. Hence, \eqref{eq:pApB} and Lemma~\ref{lemma:SlFate} yield
 \begin{equation}\label{eq:2n+1_3n}
  \bP[U_t\in \Soo'(\infty)]
  = \frac{1}{n}+\left(1-\frac{1}{n}\right)\sum_{d=1}^\infty \frac{3}{n+2}\left(\frac{n-1}{n+2}\right)^{d-1} \frac{2}{3}
  = \frac{2n+1}{3n},
 \end{equation}
 which implies convergence of the mean \eqref{eq:SooLLN} by \eqref{eq:exp}, since $|S'_\infty(\infty) \setminus S_\infty(\infty)|$ is stochastically bounded. 
 
We remark here that it can also be concluded from \eqref{eq:2n+1_3n} that 
 \begin{equation}\label{eq:23}
  \bP_t[S\ssq\Soo'(\infty)]=\frac{2}{3}
 \end{equation}
 for a sling $S\in\Slr_n(t)$. 
 This fact will be repeatedly used below in the variance computations for the central limit theorem. 

 To show almost sure convergence, we note that $\{|\Soo''(t)|\}_{t\in\IN_0}$ is a process with regenerative increments over $(T_j, j \geq 0)$, that is, the pairs 
 \[
  \big(R_j, \{ |\Soo''(T_{j-1}+t)|-|\Soo''(T_{j-1})| \}_{0\leq t\leq R_j} \big),
 \]
indexed by $j$, are independent and identically distributed (where we set $T_0=0$). 
 It then follows from the elementary renewal/reward theorem that
 \[
  \frac{\bE[|\Soo''(t)|]}{n(t+1)} \to \frac{\bE[|\Soo''(T_1)|]}{n\bE[T_1]}
 \]
 and by a suitable strong law of large numbers such as \cite[Theorem~2.54]{Serfozo09}, we also have almost sure convergence for $\frac{|\Soo''(t)|}{n(t+1)} \to \frac{\bE[|\Soo''(T_1)|]}{n\bE[T_1]}$.
 
 Since we have already established \eqref{eq:SooLLN} for convergence in expectation, we can conclude 
 \begin{equation}\label{eq:quotexpect}
  \frac{\bE[|\Soo''(T_1)|]}{n\bE[T_1]}=\frac{2n+1}{3n}
 \end{equation}
 which provides almost sure convergence in \eqref{eq:SooLLN}. 
\end{proof}

\subsection*{Central limit theorem}
We prove here Theorem \ref{theorem:Soo}(ii).
As established in the proof for the strong law of large numbers, $\{\Soo''(t)\}_{t\in\IN_0}$ is a process with regenerative increments, and therefore subject to the regenerative central limit theorem, which we cite here for convenience: 

\begin{theorem}[from {\cite[Theorem~2.65]{Serfozo09}}]\label{theorem:regCLT}
 Suppose $Z(t)$ is a process with regenerative increments over $(T_j, j \geq 0)$ such that $\bE[T_1]$ and $\bE[Z(T_1)]/\bE[T_1]$ are finite. In addition, let 
 \[
  M_j=\sup_{T_{j-1}\leq t\leq T_j} |Z(t)-Z(T_{j-1})|, \quad j\geq 1
 \]
 and assume $\bE[M_1]$ and $\bVar\left[Z(T_1)-\frac{\bE[Z(T_1)]}{\bE[T_1]}T_1\right]$ are finite and the variance is strictly positive. Then
 \begin{equation}\label{eq:regCLT}
  \frac{Z(t)-\frac{\bE[Z(T_1)]}{\bE[T_1]}t}{\sqrt{\frac{1}{\bE[T_1]}\bVar\left[Z(T_1)-\frac{\bE[Z(T_1)]}{\bE[T_1]}\right]t}} \to N(0,1)
 \end{equation}
 in distribution as $t\to\infty$.
\end{theorem}

As Svante Janson pointed out to us, the proof of this theorem given in \cite{Serfozo09} is erroneous. The argument given there can be made to work if one additionally assumes finite second moments for $M_j$ and $R_j = T_j - T_{j-1}$. Alternatively, the gap in the argument of \cite{Serfozo09} has been filled by \cite{Jan23} under even weaker assumptions, and we refer to the latter for an extended discussion. 

For $Z(t)=|\Soo''(t)|$, it is easy to see that all finiteness conditions are satisfied since the random variables $T_1, Z(T_1)$ and $M_1$ all are stochastically bounded by a constant multiple of a geometric distribution. Moreover, for $n\geq 3$ the variance is also strictly positive. With the help of \eqref{eq:quotexpect}, Theorem~\ref{theorem:regCLT} establishes that a central limit theorem holds, though evaluating the variance directly in \eqref{eq:regCLT} seems unfeasible. 

Instead, the strategy to compute the asymptotic variance will be to directly evaluate $\bVar[t^{-1/2}(\Soo'(t)-\frac{2n+1}{3}t)]$ (Lemmas~\ref{lemma:Slfate2} and \ref{lemma:SooVar}) and employ uniform integrability (Lemma~\ref{lemma:SooUI}). Here again, we will design exploration processes and compute the probability of hitting a certain set of points before an other one.

\begin{lemma}\label{lemma:Slfate2}
 Let $n=2m+1\geq 5$ and let $S_1,S_2$ be two different distinguished slings in $\Slr_n(t)$. 
 We have the following probabilities:
 \begin{center}
 \begin{tabular}{ccc}
  \toprule
  Line no. & Event & $\bP_t[\text{Event}]$ \\ \midrule \addlinespace
  (i) & $S_1,S_2 \ssq \Soo(t+1)$ & $\frac{8}{(n+4)(n+2)}$ \\ \addlinespace
  (ii) & $S_1 \ssq \Soo(t+1), S_2\ssq \Cl_n(t+1)$ & $\frac{2}{(n+4)(n+2)}$ \\ \addlinespace
  (iii) & $S_1 \ssq \Soo(t+1), S_2\ssq \Slr_n(t+1)$ & $\frac{2(n-1)}{(n+4)(n+2)}$ \\ \addlinespace
  (iv) & $S_1 \ssq \Cl_n(t+1), S_2\ssq \Slr_n(t+1)$ & $\frac{(n-1)}{(n+4)(n+2)}$ \\ \addlinespace
  (v) & $S_1,S_2 \ssq \Cl_n(t+1)$ separately  & $\frac{1}{(n+4)(n+2)}$ \\ \addlinespace
  (vi) & $S_1,S_2 \ssq \Cl_n(t+1)$ jointly  & $\frac{2}{(n+4)(n+2)}$ \\ \addlinespace 
  (vii) & $S_1,S_2 \ssq \Slr_n(t+1)$ separately & $\frac{(n-1)(n-3)}{(n+4)(n+2)}$ \\ \addlinespace
  (viii) & $S_1,S_2 \ssq \Slr_n(t+1)$ jointly & $\frac{4(n-1)}{(n+4)(n+2)}$ \\ \addlinespace  \bottomrule
 \end{tabular}
 \end{center}
 Here, we use the notation of Lemma~\ref{lemma:SlFate}. In addition, the attribute ``separately'' indicates that $S_1$ and $S_2$ remain in different connected components in $\Br'_n(t+1)$, and ``jointly'' means that they become part of a joint connected component. 
\end{lemma}
Observe that the table above covers all possible cases, except for reversing the roles of $S_1$ and $S_2$ in lines (ii)-(iv). We can then check that the probabilities sum to one. 
\begin{proof}
 We use the same exploration technique as previously. Again, the core of the proof consists in designing the right starting point and possible endpoints of the exploration, depending on the event whose probability needs to be computed. We show (i), (viii) and (iii), leaving the others to the reader. For (i), that is, the event $S_1,S_2 \ssq \Soo(t+1)$, we start exploring towards the right from $v_\infty(t)$, stopping when we reach either one of the $n$ vertices in level $t+1$ or one of the $4$ endpoints of $S_1$ and $S_2$ in level $t$. For both slings to become part of $\Soo(t+1)$, the exploration needs to reach one of those $4$ endpoints first among these $n+4$ vertices, after which it will follow along this sling and reach the other endpoint, and then among the remaining $n+2$ vertices, it needs to reach one of the two endpoints of the other slings first. From the uniformity of the matching, this happens with probability $\frac{8}{(n+4)(n+2)}$.
 An analogous argument proves line (vi) of the table, if we start the exploration from the top endpoint of $S_1$ instead. 
 
 For line (viii), start the exploration from that endpoint of $S_1$ that leads to $S_2$ before reaching a vertex of level $t+1$. 
 Following along the exploration then gives a contribution of $\frac{2n}{(n+4)(n+2)}$, since after visiting $S_2$ any of the $n$ vertices in level $t+1$ is a desired target. However, we had two possibilities for the starting vertex, and the other endpoint of $S_1$ still needs to be matched to level $t+1$ as well; rather than connecting to $\Soo(t)$. This leads to an extra factor of $\frac{2(n-1)}{n}$. 
 
 All other lines in the table require an argument involving two different explorations performed after one another. We develop the proof of line (iii): we start again from $v_\infty(t)$ and stop upon reaching level $t+1$. To visit $S_1$ but not $S_2$ on this exploration happens with probability $\frac{2n}{(n+4)(n+2)}$. Then, $S_2$ becomes part of a larger sling if and only if, when we start a second exploration from its topmost endpoint, this exploration reaches one of the $n-1$ remaining vertices in level $t+1$ out of the $n$ relevant vertices (which additionally include the lower endpoint of $S_2$). This leads to another factor of $\frac{(n-1)}{n}$. 
 
 We omit writing out the remaining arguments for lines (ii), (iv), (v), and (vii), as they are entirely analogous. 
\end{proof}

\begin{lemma}\label{lemma:SooVar}
 Let $n=2m+1$ and let $U_t$, $V_t$ be two i.i.d. vertices in $\Br_n(t)$ chosen uniformly at random, independently of $\F_t$ (and possibly equal). Then
 \begin{multline}\label{eq:UtVt}
  \bP[U_t,V_t\in\Soo'(\infty)]\\
  =\left(\frac{2n+1}{3n}\right)^2 + \frac{1}{t+1}\cdot \frac{4(n-1)(n+2)(2n+1)}{135n^2} + O(t^{-2})
 \end{multline}
 for $t\to\infty$. 
\end{lemma}
\begin{proof}
 We first fix $i,j \geq 0$ and condition on the event that $U_t$ and $V_t$ are vertices from levels $i$ and $j$, respectively.  In other words, we condition on the event $E_{i,j}=\left\{ L(U_t)=i\text{ and }L(V_t)=j\right\}$. Recall that $v_\infty(i)$ and $v_\infty(j)$ denote the right endpoint of the transverse string in $\Br'_n(i)$ and $\Br'_n(j)$, respectively.
 If $U_t$ (resp. $V_t$) lies on a sling in $\Slr_n(i)$ (resp. $\Slr_n(j)$), then we denote this sling by $S_U$ (resp. $S_V$). We distinguish two cases:
 
 \emph{Case I: $i=j$.} We first cover the scenarios in which $U_t$ and $V_t$ are not placed on different slings in $\Slr_n(i)$: first, $\bP[U_t=V_t=v_\infty(i)|E_{i,i}]=\frac{1}{n^2}$ in which case both vertices already are on the transverse string. Secondly, $\bP[U_t=v_\infty(i)\neq V_t|E_{i,i}]=\frac{n-1}{n^2}$, and if this happens then $V_t$ will be part of $\Soo'(\infty)$ if and only if $S_V\ssq\Soo'(\infty)$, which occurs with probability $\frac{2}{3}$ by \eqref{eq:23}. The event $U_t\neq v_\infty(i)=V_t$ is treated in the same way. Thirdly, $\bP[S_U=S_V, U_t\neq v_\infty(i)\neq V_t|E_{i,i}]=\frac{2(n-1)}{n^2}$, and once again the two vertices will be incorporated to $\Soo'(\infty)$ if and only if their joint sling will, so with probability $\frac{2}{3}$. In total, this yields
 \begin{multline*}
  \varrho_1:=\bP[U_t,V_t\in\Soo'(\infty) \text{ and }\\ 
  U_t,V_t \text{ are not on different slings in }\Slr_n(i)|E_{i,i}]=\frac{8n-5}{3n^2}.
 \end{multline*}

 For the remaining scenario $S_U\neq S_V$ (that is, the two vertices are placed on different slings in $\Slr_n(i)$), we have 
 \[
  \varrho_2:=\bP[U_t\neq v_\infty(i)\neq V_t\text{ and }S_U\neq S_V|E_{i,i}]=1-\frac{4n-3}{n^2}.
 \]
 (Observe that $\varrho_2=0$ for $n=3$, in which case the following discussion until \eqref{eq:815} does not influence the result, as it should be). Note that the event where $U_t$ and $V_t$ become vertices of $\Soo'(\infty)$ is now the event that eventually $S_U,S_V\ssq\Soo'(\infty)$ holds.  
 Out of all cases listed in Lemma~\ref{lemma:Slfate2}, only some are relevant here: Notably, line (i) gives the probability of $S_U,S_V\ssq\Soo'(i+1)$. We also require line (iii) in case one of the two slings becomes part of $\Soo'(i+1)$, and line (viii) in case $S_U$ and $S_V$ become parts of a joint sling in $\Slr_n(i+1)$. For both of these cases, the remaining sling is then incorporated to $\Soo'(\infty)$ with probability $\frac{2}{3}$, by \eqref{eq:23}. Thus, they become part of $\Soo'(\infty)$ with probability $\mathbf{P}[S_1,S_2 \ssq \Slr_n(t+1)$ separately$]$ $+\varrho_3$, where 
 \begin{align*}
  \varrho_3:={}& \frac{8}{(n+4)(n+2)}+\frac{2}{3}\left(\frac{2\cdot 2(n-1)}{(n+4)(n+2)} + \frac{4(n-1)}{(n+4)(n+2)}\right)\\
  ={}& \frac{16n+8}{3(n+4)(n+2)}.
 \end{align*}
 
 It remains to discuss the case where $S_U$ and $S_V$ remain in separate slings of $\Slr_n(i+1)$. Denote by $D(S_U,S_V)$ the first $d\geq 1$ such that $S_U$ and $S_V$ are not in separate slings of $\Slr_n(i+d)$. By line (vii) in the table of Lemma~\ref{lemma:Slfate2}, we have 
 \[
  \bP[D(S_U,S_V)\geq d\mid E_{i,i}, S_U \neq S_V] 
  = \left(\frac{(n-1)(n-3)}{(n+4)(n+2)}\right)^{d-1}.
 \]

 
 Conditioning on $D(S_U,S_V)$, we can repeat the argument for all cases contributing to $\varrho_3$ above, to obtain 
 \begin{align}\label{eq:p_I}
  p_I:={}& \bP[U_t,V_t\in \Soo'(\infty)\mid E_{i,i}]\notag \\
   ={}& \varrho_1 + \varrho_2 \sum_{d=1}^\infty \varrho_3\bP[D(S_U,S_V)\geq d\mid E_{i,i}, S_U \neq S_V]\notag \\
   ={}& \frac{8n-5}{3n^2} + \left(1-\frac{4n-3}{n^2}\right)\sum_{d=1}^\infty \frac{16n+8}{3(n+4)(n+2)}\left(\frac{(n-1)(n-3)}{(n+4)(n+2)}\right)^{d-1} \notag \\
   ={}& \frac{8n-5}{3n^2} + \left(1-\frac{4n-3}{n^2}\right)\cdot \frac{8}{15}\notag \\
   ={}& \frac{8n^2+8n-1}{15n^2}.
 \end{align}
 As a consequence, we point out the useful observation from the second-to-last line in \eqref{eq:p_I} that 
 \begin{equation}\label{eq:815}
  \bP_t[S_1,S_2\ssq \Soo'(\infty)]=\frac{8}{15}
 \end{equation}
 for any two distinct slings $S_1,S_2$ in $\Slr_n(t)$. 
 
 \emph{Case II: $i>j$.} (The case $i<j$ can be treated by merely switching the roles of $U_t$ and $V_t$.) We write $k=i-j$. We note first that if $D(V_t)\leq k$, where we recall that $D(V_t)$ is the minimum $d$ such that $V_t$ belongs to either the transverse string or a closed component in $B_n(L(V_t)+d)$, then whether $U_t$ ends up on $\Soo'(\infty)$ is independent from what happens to $V_t$. Accordingly, by the earlier computations based on \eqref{eq:pApB}, conditionally on $E_{i,j}$, the event $D(V_t)\leq k$ occurs with probability 
 \[
  \frac{1}{n}+\varrho_4(k):=\frac{1}{n} + \frac{n-1}{n}\sum_{d=1}^k \frac{3}{n+2} \left(\frac{n-1}{n+2}\right)^{d-1} = 1-\frac{n-1}{n}\left(\frac{n-1}{n+2}\right)^k
 \]
 where $\varrho_4(k) = \bP\left[V_t\neq v_\infty(j), D(V_t)\leq k| E_{i,j}\right]$. On the other hand, if $D(V_t)>k$ then $V_t$ belongs to a sling of $\Slr_n(t)$. Furthermore, with probability $\frac{1}{n}$ the vertex $U_t$ is the endpoint of $\Soo'(i)$, with probability $\frac{2}{n}$ it is one of the endpoints of the sling containing $V_t$ in $\Slr_n(i)$, and with probability $\frac{n-3}{n}$ it belongs to a different sling than $V_t$ in $\Slr_n(i)$. Combining this with equations \eqref{eq:2n+1_3n}, \eqref{eq:23}, and \eqref{eq:815} yields 
 \begin{align}\label{eq:p_II}
  p_{II,k}&:=\bP[U_t,V_t\in\Soo'(\infty)|E_{j+k,j}]\notag \\
   &\ = \frac{2n+1}{3n^2} + \varrho_4(k)\frac{4n+2}{9n} + \left(1-\frac{1}{n}-\varrho_4(k)\right)\left(\frac{2}{3n} + \frac{4}{3n} + \frac{8(n-3)}{15n}\right) \notag \\
   &\ = \left(\frac{2n+1}{3n}\right)^2 + \frac{4n^2 +4n-8}{45n^2}\left(\frac{n-1}{n+2}\right)^k.
 \end{align}
 
 \emph{Combining cases I and II.} Since $U_t$ and $V_t$ are chosen independently and uniformly among the vertices in $\Br'_n(t)$, the random variable $L(U_t)-L(V_t)$ has a discrete triangular distribution supported on $\{-t,\dots,-1,0,1,\dots,t\}$, i.e.
 \[
  \bP[L(U_t)-L(V_t)=k] = \frac{t+1-|k|}{(t+1)^2}.
 \]
 Hence, by the law of total probability, \eqref{eq:p_I}, and \eqref{eq:p_II} we obtain: 
 \begin{align*}
  \MoveEqLeft\bP[U_t,V_t\in\Soo'(\infty)]\\
  &=\frac{p_I}{t+1} + 2 \sum_{k=1}^t \frac{t+1-k}{(t+1)^2} \cdot p_{II,k}\\
  &\begin{multlined}
  = \frac{1}{t+1} \frac{8n^2+8n-1}{15n^2} + \left(\frac{2n+1}{3n}\right)^2 \cdot 2 \sum_{k=1}^t \frac{t+1-k}{(t+1)^2} \\
     \qquad  + 2\cdot \frac{4n^2+4n-8}{45n^2}\sum_{k=1}^t \frac{t+1-k}{(t+1)^2}\left(\frac{n-1}{n+2}\right)^k
  \end{multlined}
 \end{align*}
 Using the summation formula
 \[
  \sum_{k=1}^t \frac{t+1-k}{(t+1)^2}\cdot q^k = \frac{tq-tq^2+q^{t+2}-q^2}{(t+1)^2(1-q)^2} = \frac{1}{t+1}\frac{q}{1-q} + O(t^{-2})
 \]
 as $t\to\infty$ for $q=\frac{n-1}{n+2}$, we arrive at 
 \begin{multline*}
  \bP[U_t,V_t\in\Soo'(\infty)]=\left(\frac{2n+1}{3n}\right)^2\\
  + \frac{1}{t+1}\left(\frac{8n^2+8n-1}{15n^2}-\frac{4n^2+4n+1}{9n^2}\right.
   \left. + \frac{8n^3-24n+16}{135n^2}\right) + O(t^{-2})
 \end{multline*}
 which simplifies to the expression given in \eqref{eq:UtVt}.
\end{proof}

The next lemma gives a criterion for the uniform integrability of a process with regenerative increments.

\begin{lemma}\label{lemma:SooUI}
 Let $\{Z(t)\}_{t\geq0}$ be a process with regenerative increments over integer renewal times $T_0=0<T_1<T_2<\dots$, with $Z(0)=0$. Suppose all moments of $T_1$ and $\bE[Z(T_1)]/\bE[T_1]$ are finite. Assume further that the random variable  
 \begin{equation*}
  M_1=\sup_{0\leq t\leq T_1} \left|Z(t)\right|
 \end{equation*}
 has finite $r$-th moments for all $1\leq r<\infty$. Then
 \[
  \left\{ \left|\frac{Z(t)-\frac{\bE[Z(T_1)]}{\bE[T_1]}t}{\sqrt{t}}\right|^r \right\}_{t\in\IN}
 \]
 is a uniformly integrable sequence of random variables. 
\end{lemma}

%

\begin{proof}
 For brevity, we write $X_j=Z(T_j)-Z(T_{j-1})$, where $j\geq 1$. Note that this is an i.i.d. sequence of random variables, with $X_1=Z(T_1)$.
 
 Since $r$ is arbitrary, it suffices by \cite[Theorem~5.4.1]{Gut13} to show that the family of random variables $t^{-1/2}\big(|Z(t)-\frac{\bE[X_1]}{\bE[T_1]}t|\big)$ is uniformly bounded in $L^r$, for all sufficiently large $r$.  
 
 Recall that, for all $t \geq 1$, $N(t)=\max\{j\geq 0:T_j\leq t\}$. By Minkowski's inequality, we have  
 \begin{multline}\label{eq:SooUIproof}
  \left\|\frac{1}{\sqrt{t}}\left(Z(t)-\frac{\bE[X_1]}{\bE[T_1]}t\right)\right\|_r 
  \leq \frac{1}{\sqrt{t}}\Big\| Z(t)-Z(T_{N(t)})|\Big\|_r \\
   + \frac{1}{\sqrt{t}} \Big\| Z(T_{N(t)}) - \bE[X_1]N(t) \Big\|_r 
   + \bE[X_1]\left\|\frac{N(t)-t/\bE[T_1]}{\sqrt{t}} \right\|_r   . 
 \end{multline}
 
 
 For the first summand, note that $N(t)\leq t$ since the renewal intervals have lengths of at least 1, and consider 
 \begin{multline*}
  \bE\left[\left|Z(t)-Z(T_{N(t)})\right|^r\right] 
  = \bE\left[\sum_{j=0}^t \mathds{1}\{N(t)=j\} \left|Z(t)-Z(T_j)\right|^r \right]\\
  \leq \bE\left[ \sum_{j=0}^t \sup_{T_j\leq u\leq T_{j+1}} \left|Z(u)-Z(T_j)\right| ^r \right]
  = \bE\left[ \sum_{j=0}^t M_{j+1}^r\right] 
  = (t+1) \bE[M_1^r].
 \end{multline*}
 Hence, 
 \[
  \frac{1}{\sqrt{t}}\Big\| Z(t)-Z(T_{N(t)})|\Big\|_r \leq t^{-1/2} (t+1)^{1/r} \|M_1\|_r
 \]
 which is uniformly bounded over $t\geq 1$ for $r\geq 2$. 
 
 For the second term on the right-hand side of \eqref{eq:SooUIproof}, we use telescoping to write $Z(T_{N(t)})=\sum_{i=1}^{N(t)}X_i$, which is a sum of i.i.d. random variables. Together with $N(t)\leq t$, we obtain for $r\geq2$
 \begin{equation}\label{eq:SooUIDoob}
  \frac{1}{\sqrt{t}} \Big\| Z(T_{N(t)}) - \bE[X_1]N(t) \Big\|_r 
  \leq \frac{1}{\sqrt{t}} \left\| \sum_{i=1}^{N(t)} (X_i - \bE[X_i]) \right\|_r 
  \leq \frac{1}{\sqrt{t}} \left\| \max_{k\leq t} \left|\sum_{i=1}^k (X_i-\bE[X_i]) \right|\right\|_r.
 \end{equation}
Observe that $t \mapsto \sum_{i=1}^t (X_i-\bE[X_i])$ forms a martingale with respect to the natural filtration of $\{Z(t)\}_{t\geq0}$. Hence, we can use Doob's maximal inequality, \cite[Theorem~10.9.4]{Gut13}, on the right-hand side of \eqref{eq:SooUIDoob} to obtain
 \begin{align*}
  \frac{1}{\sqrt{t}} \Big\| Z(T_{N(t)}) - \bE[X_1]N(t) \Big\|_r 
  &\leq \frac{1}{\sqrt{t}}\cdot \frac{r}{r-1} \left\| \sum_{i=1}^t (X_i-\bE[X_i]) \right\|_r\\
  &\leq \frac{1}{\sqrt{t}}\cdot \frac{r}{r-1} \left( B_r t^{r/2} \bE\left[\big|X_1-\bE[X_1]\big|^r \right] \right)^{1/r}
 \end{align*}
 where the final inequality follows from the Marcinkiewicz-Zygmund inequalities with a constant $B_r$ only depending on $r$, cf. \cite[Corollary~3.8.2]{Gut13}. Simplifying the right-hand side, we obtain
  \[
  \frac{1}{\sqrt{t}} \Big\| Z(T_{N(t)}) - \bE[X_1]N(t) \Big\|_r 
  \leq \frac{r}{r-1}B_r^{1/r} \left\| X_1 - \bE[X_1] \right\|_r
 \]
 which is a uniform bound. 
 
 The final term on the right-hand side of \eqref{eq:SooUIproof} is uniformly bounded too, since  
 \[
  \left|\frac{N(t)-t/\bE[T_1]}{\sqrt{t}}\right|^r
 \]
 is uniformly integrable by \cite[Theorem 1]{CHL79} (where we applied the theorem for $\alpha=0$). We thus conclude the desired boundedness for the right-hand side of \eqref{eq:SooUIproof}, which finishes the proof.
\end{proof}

\begin{proof}[Proof of Theorem~\ref{theorem:Soo}, central limit theorem]
 We already established the validity of a central limit theorem for $|\Soo''(t)|$ with Theorem~\ref{theorem:regCLT} and it remains to determine the term for the variance. Thus let $X$ be $N(0,1)$-distributed. Since $Z(t)=|\Soo''(t)|$ satisfies all conditions for Lemma~\ref{lemma:SooUI} (cf. the discussion below Theorem~\ref{theorem:regCLT}), setting $r=2$ implies 
 \begin{equation}\label{eq:Varconvc}
  \bVar\left[\frac{|\Soo''(t)|-\frac{2n+1}{3}t}{\sqrt{c^2t}}\right] \to \bVar[X]=1
 \end{equation}
 as $t\to\infty$, for a suitable constant $c^2$, where $c^2$ is the variance in Theorem~\ref{theorem:regCLT}. At the same time, observe that 
 \begin{align*}
  \frac{\bE[|\Soo''(t)|^2]}{n^2(t+1)^2} 
   &= \frac{1}{n^2(t+1)^2}\sum_{u,v}\bE[\mathds{1}\{u,v\in\Soo''(t)\}]\\ 
   &= \bP[U_t,V_t\in\Soo'(\infty)],
 \end{align*}
 where $U_t$ and $V_t$ are independently chosen random vertices. By virtue of Lemma~\ref{lemma:SooVar} and \eqref{eq:2n+1_3n}, we thus obtain 
 \begin{align*}
  \bVar\left[\frac{|\Soo''(t)|-\frac{2n+1}{3}t}{\sqrt{t}}\right] 
  &= \frac{1}{t}\left(\bE[|\Soo(t)''|^2] - \bE[|\Soo(t)''|]^2\right) \\
  &= \frac{4(n-1)(n+2)(2n+1)}{135} + O(t^{-1}).
 \end{align*}
 Comparing this to \eqref{eq:Varconvc} provides the desired value of $c^2$, and concludes the proof of convergence in distribution for \eqref{eq:SooCLT}. Convergence of moments now follows from the uniform integrability that was established in Lemma~\ref{lemma:SooUI}.
\end{proof}

\subsection*{The shape of the transverse string} 
This subsection is dedicated to the proof of Theorem~\ref{theorem:Soolocal}. 

\begin{proof}[Proof of Theorem~\ref{theorem:Soolocal}] \emph{(i)}
We want to investigate the number of vertices of $\Soo(\infty)$ at a given level $t$. Consider an exploration of the transverse string $\Soo(\infty)$ starting from level 0. Suppose $\tau\leq t$ and enumerate the vertices in $\Soo(\infty)\cap \mathbf{X}_t$ according to the order in which they are reached by the exploration of $\Soo(\infty)$, say by $v_1=v_\infty(t),v_2,\dots,v_r=w_\infty(t)$; where $r\leq n$ is odd. 
 
 Recall from the construction of $\Br_n(t)$ that $\Pi_u$ denotes the random uniform matching in layer $u$. Whether or not an exploration started in $v\in\mathbf{X}_t$ to the right eventually returns to level $t$ or not depends only on $\{\Pi_s, s \geq t+1\}$. In particular, $w_\infty(t)$ is independent of $\left\{\Pi_s, 1\leq s\leq t\right\}$ and is uniform on $\mathbf{X}_t$.
  Moreover, conditioned on $v_1,\dots,v_{2i}$, the random vertex $v_{2i+1}$ is independent of $\left\{\Pi_s, s \geq t+1\right\}$ and conditioned on $v_1,\dots,v_{2i-1}$, the random vertex $v_{2i}$ is independent of $\left\{\Pi_s, 1\leq s\leq t\right\}$. 
  
  We can rewrite the event $\{V_t=2j+1\}$ as $\{v_1,\dots,v_{2j}\neq w_{\infty}(t), v_{2j+1}=w_\infty(t)\}$. By the aforementioned conditional independencies, and since the involved conditional distributions are uniform, we hence obtain for $0\leq j\leq m$:
 \begin{align*}
  \bP[V_t=2j+1]
  &= \bP[v_{2j+1}=w_\infty(t)\mid v_1,\dots,v_{2j}\neq w_\infty(t)]\\
  &\qquad \cdot\prod_{i=0}^{j-1} \bP[v_{2i+1}\neq w_\infty(t)\mid v_1,\dots,v_{2i}\neq w_\infty(t)]\\
  &= \frac{1}{2m+1-2j} \prod_{i=0}^{j-1} \frac{2m-2i}{2m+1-2i}
 \end{align*}
 which simplifies to \eqref{eq:Vdist}. We remark that it is easy to confirm that the distribution given in \eqref{eq:Vdist} is truly a probability distribution on $\{1,3,\dots,2m+1\}$ -- induction over $m$ suffices. In the same way, one finds that $\bE[V]=\frac{4m+3}{3}=\frac{2n+1}{3}$, and thus Theorem~\ref{theorem:Soolocal} can be used to give an alternative proof of \eqref{eq:SooLLN}.

 \emph{Proof of Theorem~\ref{theorem:Soolocal} (ii)}
We now consider the edges of $\Soo(\infty)$ connecting two fixed consecutive levels of $\Br'_n$. For integers $j,k,\ell\in\{0,\dots,m\}$, we show that 
 \begin{equation}\label{eq:Soolocalproof}
  \bP[V_t=2j+1, V_{t+1}=2k+1, E_t=2\ell+1] = \binom{j}{\ell}\binom{k}{\ell}\frac{\fdf{2m}_j \fdf{2m}_k}{\fdf{4m+1}_{j+k+1}}. 
 \end{equation}
 The distribution for $E_t$ in \eqref{eq:Edist} is then obtained by summing over all values of $j$ and $k$.
 
 To obtain \eqref{eq:Soolocalproof}, consider an exploration of $\Soo(\infty)$ starting in $\mathbf{X}_0$. Label the vertices in $\Soo(\infty)\cap \mathbf{X}_t$ by $v_1,v_2,\dots$ in the order in which they are visited by the exploration, and label the vertices in $\Soo(\infty)\cap \mathbf{X}_{t+1}$ analogously by $w_1,w_2,\dots$. We say that a right-transition (resp. left-transition) occurs whenever the exploration goes from $\mathbf{X}_t$ to $\mathbf{X}_{t+1}$ (resp. vice-versa). We will compute the number of matchings on $\mathbf{X}_t \cup \mathbf{X}_{t+1}$ such that $V_t=2j+1, V_{t+1}=2k+1, E_t=2\ell+1$.
 
 For parity reasons, right-transitions only occur after a vertex $v_i$ with odd $i$, whereas left-transitions only occur after a vertex $w_i$ with even $i$. The event $E=2\ell+1$ requires the exploration to make $\ell+1$ right- and $\ell$ left-transitions in layer $t$; however, for $V_t=2j+1$ we also need $v_{2j+1}=w_\infty(t)$, and hence the right-transition after $v_{2j+1}$ is forced. Hence, the endpoints on $\mathbf{X}_t$ of the remaining $\ell$ right-transitions are to be chosen among $v_1, v_3, \ldots, v_{2j-1}$, contributing the factor $\binom{j}{\ell}$ in \eqref{eq:Soolocalproof}. An analogous argument for left-transition yields the factor $\binom{k}{\ell}$.
 These instants for transitions being chosen, we now choose the points of $\mathbf{X}_{t} \cap \mathbf{X}_{t+1}$ that are explored. That is, we want a list of points of $\mathbf{X}_{t} \cap \mathbf{X}_{t+1}$ that respects the relative order of the $v_i$ and $w_i$, starts with $v_1$, ends with $w_{2k+1}$ and respects the instants of transitions. We wil call a list satisfying these properties a \emph{merging}.
 
 For the exploration to realize a merging $M$, all transitions along edges in layer $t$ must lead to the correct level (either $\mathbf{X}_t$ or $\mathbf{X}_{t+1}$). For any $i$, the vertex $v_{2i}$ can be chosen among the $2m-2i+2$ vertices in $\mathbf{X}_t\setminus\{v_1,\dots,v_{2i-1}\}$. For any $i'$, the vertex $w_{2i'+1}\neq w_\infty(t+1)$ can be chosen among the $2m-2i'$ vertices in $\mathbf{X}_{t+1}\setminus\{w_1,\dots,w_{2i'},w_\infty(t+1)\}$ (note that $w_\infty(t)$ is independent of $\{\Pi_s, s\leq t\}$, and that the unmentioned vertices $v_{2i+1}$ and $w_{2i'}$ are independent of $\Pi_t$ conditioned on the previous history of the exploration). Any given merging $M$ satisfying the previous conditions imposes exactly $j+k+1$ edges in layer $t$. Therefore, we get
 \[
  \bP[\text{exploration realizes $M$}] = \frac{\fdf{2m}_j \fdf{2m}_k}{\fdf{4m+1}_{j+k+1}}
 \]
 and \eqref{eq:Soolocalproof} follows. 
 
 
 \emph{Proof of Theorem~\ref{theorem:Soolocal} (iii)}
 The claim of Theorem~\ref{theorem:Soolocal}(iii) would follow from computing $\bE[E_t|\tau\leq t]$ using \eqref{eq:Edist}, since $A_t = \sum_{s \leq t-1} E_s$; however, we were not able to evaluate this sum directly. Thus we will instead use a direct probabilistic argument to obtain the desired expectation. For this, we once again work with the modified Brauer diagram $\Br'_n$.
 
 Recall that the vertices $x_{t,j}$ in level $t$ are ordered by the index $j$ from top to bottom.
 This extends to slings in $\Slr_n(t)$ by ordering them according to the position of their upper endpoint in level $t$ (recall that, in $\Br'_n(t)$, there are $m$ slings at each level $t\geq 0$). We denote these slings by $S_1,\dots,S_m$ according to this ordering, and we denote the upper endpoint of $S_k$ by $y_k$ and the lower endpoint by $z_k$, for $k=1,\dots,m$. 
 
 
 For brevity, we say a sling $S_k$ in $\Slr_n(t)$ \emph{propagates} if there is a sling $S$ in $\Slr_n(t+1)$ such that $S_k$ is the topmost sling in $\Slr_n(t)$ that is part of $S$. The probability that $S_k$ propagates equals the probability that an exploration starting from $z_k$ and going to the right hits $\mathbf{X}_{t+1}$ before it hits $y_1,\dots,y_k,z_1,\dots,z_{k-1}$ or $v_\infty(t)$ (we stop the exploration when it hits $\mathbf{X}_{t+1}$), and that a second exploration starting afterwards from $y_k$ to the right hits $\mathbf{X}_{t+1}$ before it hits $y_1,\dots,y_{k-1},z_1,\dots,z_{k-1}$ or $v_\infty(t)$. Using \eqref{eq:23} for the resulting sling after propagation, we obtain
 \begin{multline*}
  \bE\left[\mathds{1}\left\{S_k\text{ propagates}, S_k\ssq \Soo'(\infty)\right\}\right]\\ =\frac{2}{3}\cdot\frac{n}{n+2k}\cdot \frac{n-1}{n+2k-2}
  = \frac{2}{3}\left(\frac{m(2m+1)}{2m+2k-1} - \frac{m(2m+1)}{2m+2k+1}\right).
 \end{multline*}
 Denoting by $X_t$ the total number of slings in $\Slr_n(t)$ that propagate and eventually become part of the transverse string, we get from the above that 
 \[
  \bE[X_t] = \sum_{k=1}^m \bE\left[\mathds{1}\left\{S_k\text{ propagates}, S_k\ssq \Soo'(\infty)\right\}\right] = \frac{4m^2}{12m+3},
 \]
 by telescoping. Observe next that the equality $E_t=2X_t+1$ holds deterministically, and that therefore, since $A_t = \sum_{s=0}^{t-1} E_s$:
 \[
  \frac{\bE[A_t]}{t} = \frac{1}{t} \sum_{s=0}^{t-1} (2\bE[X_s]+1) \to 1+\frac{8m^2}{12m+3}.
 \]
 Since $\left\{A_t\right\}_{t\geq 0}$ is a process with regenerative increments, we can argue as in the proof of the strong law of large numbers for $|\Soo(t)|$ to obtain the almost sure convergence for $A_t$ in \eqref{eq:ABLLN}. This, together with \eqref{eq:SooLLN}, yields in turn the convergence for $B_t$ in \eqref{eq:ABLLN} and hence joint convergence, since $B_t=|\Soo(t)|-A_t-1$. 
\end{proof}

\section{Counting components -- proofs}
\label{sec:proofsI}

The goal of this section is to prove Theorem~\ref{theorem:CC}, concerning the number $C_n(t)$ of closed loops in $\Br'_n(t)$ as $t \rightarrow \infty$ . The central argument for its proof is that $C_n(t)$ is close to a sum $\sum_{s=0}^t Y_n(s)$ of i.i.d. random variables. All claims of the theorem then follow from the standard theory about such sums -- see e.g. \cite[Theorems $6.6.1$, $6.10.2$, $7.1.1$ and $7.5.1$]{Gut13}. 

We begin by switching from $\Br_n(t)$ to the modified Brauer diagram $\Br'_n(t)$, see Remark~\ref{rem:modified} and Section~\ref{sec:RT}. Define $Y_n(t)$ for $t\geq 0$ to be the difference between the number of closed loops in the modified diagrams $\Br'_n(t+1)$ and in $\Br'_n(t)$. Observe that the loops counted by $Y_n(t)$ are precisely those that came from slings in $\Slr_n(t)$, with the help of layer $t+1$. 
Since the matchings in each layer are chosen uniformly and independently from one another, the random variables $(Y_n(t), t\geq 0)$ are i.i.d.. Their exact distribution is described by Lemma~\ref{lemma:CCYdistr} below.  

Finally, the difference between the true count $C_n(t)$ for the number of closed loops in $\Br_n(t)$, and $\sum_{s=0}^t Y_n(s)$ is bounded by the total number of slings and transverse strings in the unmodified $\Br_n(t)$. Hence, by Remark~\ref{rem:CC},
\[
 \left|C_n(t)-\sum_{s=0}^t Y_n(s)\right|\leq n,
\]
that is, this difference is small enough to not affect the limit laws in Theorem~\ref{theorem:CC}.

\begin{lemma}\label{lemma:CCYdistr}
Fix $t\geq 0$. Let $Y_n(t)$ denote the number of closed loops in the modified $\Br'_n(t+1)$ that contain slings in $\Slr_n(t)$. If $n$ is even, then $Y_n(t)$ has the probability generating function 
 \begin{equation}\label{eq:pgfYeven}
  G_{Y_n(t)}(x)=\frac{(n-1)!!}{(2n-1)!!}\prod_{k=1}^m (x+n+2(k-1)). 
 \end{equation}
In particular,
 \begin{equation}\label{eq:evenYmoments}
  \bE[Y_n(t)] = \sum_{k=1}^m \frac{1}{n+2k-1} \quad \text{ and } \quad 
  \bVar[Y_n(t)]= \sum_{k=1}^m \frac{n+2(k-1)}{(n+2k-1)^2}.
 \end{equation}
 If $n$ is odd, then $Y_n(t)$ has the probability generating function 
 \begin{equation}\label{eq:pgfYodd}
  G_{Y_n(t)}(x)=\frac{n!!}{(2n-1)!!}\prod_{k=1}^m (x+n+2k-1). 
 \end{equation}
In particular,
 \begin{equation}\label{eq:oddYmoments}
  \bE[Y_n(t)] = \sum_{k=1}^m \frac{1}{n+2k} \quad \text{ and } \quad 
  \bVar[Y_n(t)]= \sum_{k=1}^m \frac{n+2k-1}{(n+2k)^2}.
 \end{equation}
\end{lemma}

\begin{proof}
 As in the proof of Theorem~\ref{theorem:Soolocal}(iii) above, we order the $m$ slings in $\Slr_n(t)$ according to the position of their upper endpoint in level $t$. We denote these slings by $S_1,\dots,S_m$ according to this ordering, and we denote the upper endpoint of $S_k$ by $y_k$ and the lower endpoint by $z_k$, for $k=1,\dots,m$. The endpoint of $\Soo(t)$ in level $t$ will be denoted by $v_\infty(t)$, if a transverse string exists. 
 Finally, we can order all loops which close in layer $t+1$ according to the topmost sling in $\Slr_n(t)$ that is contained in that loop. Notice that this provides an injection from the set of loops that closed in layer $t+1$ to $\Slr_n(t)$. 
 
 For brevity, we will say that a sling $S_k$ in $\Slr_n(t)$ \emph{forms a loop} if it is the topmost sling in a closed loop in $\Br'_n(t+1)$ (among the slings of $\Slr_n(t)$). 
 
 In case $n$ is even, each sling either forms a loop, or is part of a closed loop in $\Br'_n(t)$ but not the topmost sling in that loop, or is part of a larger sling in $\Slr_n(t+1)$. 
 We fix $k$ and explore the continuation of $S_k$ in $\Br'_n(t+1)$ by starting at $y_k$ and going to the right, continuing along the connected component. Consider now the set consisting of all vertices in level $t+1$ together with $y_1,\dots,y_{k-1}$ and $z_1,\dots,z_k$. Stop the exploration once we reach one of these vertices. Since the matching in layer $t+1$ is uniform, each of these $n+2k-1$ vertices is equally likely to be the endpoint of the exploration. For $S_k$ to form a loop, we must have reached $z_k$ first, which happens with probability $\frac{1}{n+2k-1}$. In all other cases, $S_k$ is either not the topmost sling of its component, or part of a larger sling in $\Slr_n(t+1)$. 
 
 Let  $Y_{n,k}(t)$ be the indicator random variable for the event that $S_k$ forms a loop. By the preceeding discussion, $Y_{n,k}(t)\sim \Ber\big(\frac{1}{n+2k-1}\big)$, and we claim that the $Y_{n,k}$ are independent from one another, for $k=1,\dots,m$, which implies \eqref{eq:pgfYeven}. 
 To see this, start the exploration described above from $y_m$. Afterwards start a new exploration from $y_{m-1}$, and afterwards from $y_{m-2}$, and so on, with the final exploration starting from $y_1$.
Independently of whether or not $S_k$ formed a loop, the exploration started at $y_{k'}$ with $k'<k$ is clearly stopped at one of the vertices $y_1,\dots,y_{k'-1}, z_1,\dots, z_{k'}$ or in one of the vertices in level $t+1$ that have not been reached by a previous exploration, whose distribution is uniform on this set.
 
 The claims about expectation and variance in \eqref{eq:evenYmoments} follow immediately from this. 
 
 The argument for the case where $n$ is odd uses the same exploration technique, but we additionally stop an  exploration starting from $y_k$ if we reach $v_\infty(t)$ (in which case $S_k$ becomes a part of $\Soo(t+1)$). Thus, with $Y_{n,k}(t)$ being defined verbatim as in the even case, we have $Y_{n,k}\sim \Ber\big(\frac{1}{n+2k}\big)$, and independence can be shown in the same way as above, so we obtain \eqref{eq:pgfYodd} and \eqref{eq:oddYmoments}. 
\end{proof}
 
\section{Components by shape -- Proofs}
\label{sec:proofsIII}

Our goal here is to prove Theorem \ref{theorem:shape}.
We begin this section by first ensuring that the map associating a weak shape $\mb a$ to a strong shape $\shape$ is indeed well-defined, as claimed in Definition~\ref{def:shape}. In doing so, we will also obtain a precise characterisation of all words over $\{A,B\}$ that can occur as strong shapes. To this end, write 
\[
 \{A,B\}^* = \{\varepsilon\} \cup \bigcup_{\ell=1}^\infty \{A,B\}^{\ell}
\]
for the set of all finite words over the alphabet $\{A,B\}$ (where $\varepsilon$ stands for the empty word). Consider the automaton in Figure~\ref{fig:automaton}.
\begin{figure}[!ht]
 \centering
 \includegraphics[width=0.9\textwidth]{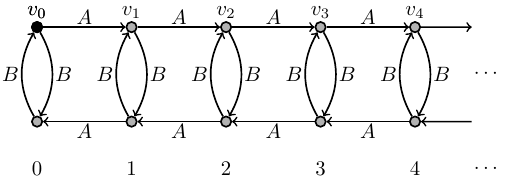}
 \caption{An infinite deterministic automaton, where the unique starting and accepting state is the black vertex $v_0$. Just as for the vertices in $\Br_n(t)$, we think of the states as being arranged in levels, with levels being enumerated as in the bottom of the figure.}\label{fig:automaton}
\end{figure}

For any word $w \in \{ A, B \}^*$, read $w$ and follow the corresponding edges of the automaton starting from $v_0$. We say that a word $w \in \{ A,B\}^*$ is accepted by the automaton if after reading $w$ we end up in the accepting state, $v_0$.

\begin{lemma}\label{lemma:automaton}
 A non-empty word $\shape\in\{A,B\}^*$ describes a strong shape for $\Br_n$ if and only if it can be read off from the automaton in Figure~\ref{fig:automaton} by starting and ending at $v_0$, such that no state in the automaton is visited more than $m$ times. Moreover, if $\shape$ is a strong shape, then any closed loop with strong shape $\shape$ has weak shape $(a_k)_{k\geq0}$ where $a_k$ (for any $k$) is the number of times the automaton visits the vertex $v_k$ when reading $\shape$. 
\end{lemma}

\begin{proof}
 Let $K$ be a closed loop in $\Br_n(t)$ with strong shape $\shape\in\{A,B\}^*$, and leftmost level $s$. Recall that this means that the exploration described in Definition~\ref{def:shape} produces $\shape$ when writing down at every step whether the upcoming edge goes across levels or bends around. Comparing the exploration process with the tracking of $\shape$ in the automaton, it is easily verified that the automaton is at a state in level $k\geq0$ if and only if the exploration is in a vertex of level $s+k$, and that the automaton is at a state in the top row of Figure~\ref{fig:automaton} if and only if the exploration process passes through the current vertex coming from the left. 
 It thus follows that $\shape$ is accepted by the automaton, that no state is visited more than $m$ times, and that $a_k$ is given by the number of visits to the top (or equivalently, bottom) state in level $k$.  
 
 If conversely $\shape\in\{A,B\}^*$ is a finite word accepted by the automaton without visiting any state more than $m$ times then we can explicitly construct a component whose exploration gives $\shape$: Start in the top vertex in some level $s$, and (except for the final step) draw an edge to the topmost available vertex at every step, in such a way that the sequence of $A$'s and $B$'s coincides with $\shape$. In the last step, connect the current vertex to the starting vertex in layer $s$. 
\end{proof} 

Observe that we can count with the help of this automaton the number of strong shapes with a given weak shape.

\begin{lemma}\label{lemma:gammana}
 Given a weak shape $\mb{a}=(a_k)_{k\geq0}$ for $\Br_n$, there are 
 \[
  \gamma_n^{\mb a} = \prod_{i\geq1} \binom{a_{i-1}+a_i-1}{a_i}
 \]
 strong shapes with weak shape $\mb a$.
\end{lemma}


\begin{proof}
 We will proceed via induction over $\str(\mb a)$, and show that $\gamma_n^{\mb a}$ counts the number of closed walks starting (and ending) in $v_0$ in the automaton of Figure~\ref{fig:automaton} with $a_k$ visits to $v_k$ for all $k \geq 0$. The claim then follows from Lemma~\ref{lemma:automaton}. 
 
 For $\str(\mb a)=0$, we evaluate $\gamma_n^{\mb a} = 1$.
 In the automaton, the only walk fulfilling all corresponding requirements is the unique walk following the arrows between the two states in level 0. 
 
 Suppose now that the statement holds for all weak shapes of stretch $r\geq0$, and let $\mb a$ be a weak shape having stretch $r+1$. Any walk realizing $\mb a$ in the automaton decomposes as follows: whenever the walk reaches the top state in level $r$, it can either move on to the top state in level $r+1$ and then move between the two states in level $r+1$ before continuing to the bottom state in level $r$, or it can skip the excursion to level $r+1$ and move directly from top to bottom in level $r$. Thus, any walk with stretch $r+1$ for $\mb a=(a_0,a_1,\dots,a_r,a_{r+1},0,\dots)$ is uniquely determined by a walk of stretch $r$ for $\mb{a'}=(a_0,a_1,\dots,a_r,0,0,\dots)$ and a sequence of non-negative integers $j_1,\dots,j_{a_r}$ such that $j_1+\dots+j_{a_r}=a_{r+1}$, where $j_k$ is the number of visits to the top state in level $r+1$ between the $k$-th visit to the top state in level $r$ and the $k$-th visit to the bottom state in level $r$. By the induction hypothesis, we thus obtain
 \[
  \#(\text{walks for }\mb a) 
  = \sum_{j_1+\dots+j_{a_r}=a_{r+1}} \gamma_n^{\mb{a'}}
  = \sum_{j_1+\dots+j_{a_r}=a_{r+1}} \prod_{i=1}^r \binom{a_{i-1}+a_i-1}{a_i}.
 \]
 The sum runs over all compositions of $a_{r+1}$ into $a_r$ non-negative summands, and these compositions are counted by $\binom{a_r+a_{r+1}-1}{a_{r+1}}$. Hence, the above simplifies to
 \[
  \#(\text{walks for }\mb a) = \prod_{i=1}^{r+1} \binom{a_{i-1}+a_i-1}{a_i} = \gamma_n^{\mb a},
 \]
 thereby concluding the proof.
\end{proof}

\begin{remark}
 As S. Wagner pointed out to us, the expression for $\gamma_n^{\mb a}$ in Lemma~\ref{lemma:gammana} also counts the number of Dyck paths with $a_k$ up-steps from height $k$ to $k+1$, for $k\geq0$, see also \cite[Proposition 10]{Flajolet80}. The bijection works by encoding Dyck paths in a relabelled version of the automaton in Figure~\ref{fig:automaton}, where transitions ending in an upper state are labelled $U$, and transitions ending in a lower state are labelled $D$. Now start by appending an extra up-step to the Dyck path (at the end of the path), and ignore its first step. Then, convert the steps of the Dyck path to transitions in the automaton beginning in the starting state according to the labelling; so up-steps correspond to $U$-transitions, and down-steps correspond to $D$-transitions. 
\end{remark}

In what follows for the proof of Theorem~\ref{theorem:shape}, our computations will frequently invoke the known identity
\begin{equation}\label{eq:disccalc}
 \sum_{k=1}^n (k-1)_a = \frac{1}{a+1}(n)_{a+1},
\end{equation}
which is obtained from a telescoping sum after observing that 
\[
 (k)_{a+1}-(k-1)_{a+1} = (k-(k-a-1))(k-1)_a = (a+1)(k-1)_a.
\]

We now compute expecations, variances and covariances of the indicators of presence of components of given shape. 

\begin{lemma}\label{lemma:EYn}
 Let $\shape$ and $\shape'$ be strong shapes for $\Br_n(\cdot)$ with weak shapes $(a_k)_{k\geq0}$ and $(a'_k)_{k\geq0}$, respectively. Set $b_0=a_0$ and $b_i=a_i+a_{i-1}$ for $i\geq1$, as in Theorem~\ref{theorem:shape}. Let $Y_{n,k}^\shape(s)$ be the indicator random variable for the event that the exploration process described in Definition~\ref{def:shape} started at the $(n-k)$-th vertex of a fixed level $s$ yields $\shape$. Define $b'_i$ and $Y_{n,k}^{\shape'}(s)$ analogously. Then 
 \begin{equation}\label{eq:Ynkshape}
  \bE\left[Y_{n,k}^\shape(s)\right]=(k-1)_{2a_0-1} \frac{\prod_{i\geq 1} (n)_{2a_i}}{\prod_{j\geq 0} \fdf{2n-1}_{b_j}}
 \end{equation}
 and, for $k'<k$,
 \begin{equation}\label{eq:Ynkshapes1}
  \bE\left[Y_{n,k}^\shape(s)Y_{n,k'}^{\shape'}(s)\right] = (k-2a'_0-1)_{2a_0-1}(k'-1)_{2a'_0-1} \frac{\prod_{i\geq 1} (n)_{2a_i+2a'_i}}{\prod_{j\geq 0} \fdf{2n-1}_{b_j+b'_j} }.
 \end{equation}
 Moreover, for $h>0$ we have (for any values of $k,k'$)
 \begin{multline}\label{eq:Ynkshapes2}
  \bE\left[Y_{n,k}^\shape(s)Y_{n,k'}^{\shape'}(s+h)\right]\\
  = (k-1)_{2a_0-1}(k'-1)_{2a'_0-1} \frac{(n-2a_0')_{2a_h}\prod_{h\neq i\geq 1} (n)_{2a_i+2a'_{i-h}}}{\prod_{j\geq 0} \fdf{2n-1}_{b_j+b'_{j-h}} }
 \end{multline}
 where we extend $a'_i=b'_i=0$ for negative indices $i$. 
\end{lemma}

Recall that $(n)_k=0$ for $k>n$, thus the expectations in Lemma~\ref{lemma:EYn} vanish whenever we encounter a configuration of shapes that requires more than $n$ vertices in a single column. 

\begin{proof}
 Let $K$ be a closed loop whose exploration starts at the $(n-k)$-th vertex of a fixed level $s$ (assuming such a component exists), that is to say, $x_{s,n-k}$ is the topmost vertex in the leftmost level intersecting $K$. 
 Now, for all levels $s+i$ with $i\geq 1$ there are $(n)_{2a_i}$ ways to select the $2a_i$ vertices in the order in which the exploration process visits them, starting from $x_{s,n-k}$ and going to the right. Given that $x_{s,n-k}$ is the topmost vertex of $K$ at level $s$, it remains to select the order for the remaining $2a_0-1$ vertices among the $k-1$ vertices below $v_0$. Thus $(k-1)_{2a_0-1}\prod_{i\geq 1} (n)_{2a_i}$ is the number of possible realisations of the strong shape $\shape$ when starting in $v_0$. 
 Since $b_j$ for $j\geq0$ counts the number of edges belonging to $K$ in layer $s+j$, we obtain that $\left(\prod_{j\geq0} \fdf{2n-1}_{b_j}\right)^{-1}$ is the probability for any of the possible realisations of $\shape$ to occur, and conclude \eqref{eq:Ynkshape}.
 
 
 To obtain \eqref{eq:Ynkshapes1}, we first trace the exploration of a closed loop $K'$ for the strong shape $\shape'$, and afterwards run a second exploration for $K$ on the complement of $K'$. Note that since $k'<k$, all vertices of $K'$ in level $s$ are below the starting vertex for the second exploration. The same counting arguments as for \eqref{eq:Ynkshape} therefore yields
 \begin{multline*}
  \bE\left[Y_{n,k}^\shape(s)Y_{n,k'}^{\shape'}(s)\right] \\
  = (k-2a'_0-1)_{2a_0-1}(k'-1)_{2a'_0-1}\frac{\prod_{i\geq 1} (n-2a'_i)_{2a_i}(n)_{2a'_i}}{\prod_{j\geq 0} \fdf{2n-1-2b'_j}_{b_j}\fdf{2n-1}_{b'_j}},
 \end{multline*}
 where the terms in the products on the right-hand side simplify to the form given in \eqref{eq:Ynkshapes1}. 
 
 Finally, the proof of \eqref{eq:Ynkshapes2} mimics the one for \eqref{eq:Ynkshapes1}: first consider the exploration of $K'$ starting at level $s+h$, then the exploration of $K$ will run in the complement of $K'$. All computations are analogous up to the additional index shift between the two shapes. 
\end{proof}

For the proof of Theorem~\ref{theorem:shape}, we once again employ the modified random Brauer diagram $\Br'_n(t)$ as introduced in Remark~\ref{rem:modified}. This introduces at most $m$ additional components of a given strong or weak shape, and is therefore without effect on the results of Theorem~\ref{theorem:shape}.

Denote by $C_n^\shape(t)$ the number of components in the modified random Brauer diagram $\Br'_n(t)$ with a given strong shape $\shape$ and let $r$ be the stretch of $\shape$. Thus 
\[
 C_n^\shape(t) = \sum_{s=0}^{t-r-1} \sum_{k=1}^n Y_{n,k}^\shape(s)
\]
(note that for a component with stretch $r$ to be in $\Br'_n(t)$, it needs to start at the latest in level $t-r-1$),
and it will be convenient to write $Y_n^\shape(s) = \sum_{k=1}^n Y_{n,k}^\shape(s)$ for the number of components in $\Br'_n(t)$ that start in at level $s$ and have strong shape $\shape$. Using \eqref{eq:disccalc} and \eqref{eq:Ynkshape}, we obtain
\begin{equation}\label{eq:Ynshape}
 \bE\left[Y_n^\shape(s)\right] 
 = \sum_{k=1}^n (k-1)_{2a_0-1}\frac{\prod_{i\geq 1} (n)_{2a_i}}{\prod_{j\geq 0} \fdf{2n-1}_{b_j}} 
 = \frac{1}{2a_0}\prod_{i\geq 0} \frac{(n)_{2a_i}}{\fdf{2n-1}_{b_i}}
\end{equation}
and thus $\bE\left[Y_n^\shape(s)\right]=\mu_n^\shape$ by a comparison with \eqref{eq:sshapeLLN}.

\begin{lemma}\label{lemma:shapeVar}
 For a strong shape $\shape$ with stretch $r\geq 0$ in $\Br'_n(t)$, we have 
 \begin{equation}\label{eq:shapeVar}
  \bVar\left[C_n^\shape(t)\right] = (t-r)\mb K_\shape(0) + 2\sum_{h=1}^{t-r-1} (t-r-h) \mb K_\shape(h),
 \end{equation}
 where $\mb K_\shape(h)=\bCov\left[Y_n^\shape(s),Y_n^\shape(s+h)\right]$ for any $0\leq s\leq t-r-1$ is the auto-covariance function of $\shape$, given by 
 \begin{equation}\label{eq:autocov}
  \mb K_\shape(h) =   
     \begin{dcases}
      \frac{1}{4a_0^2}\prod_{i\geq0}\frac{(n)_{4a_i}}{\fdf{2n-1}_{2b_i}} +\mu_n^\shape - \big(\mu_n^\shape\big)^2 & \text{ if $h=0$}\\
      \frac{1}{4a_0^2} \prod_{i\geq0} \frac{(n)_{2a_i+2a_{i-h}}}{\fdf{2n-1}_{b_j+b_{j-h}}} - \big(\mu_n^\shape\big)^2 & \text{ if $1\leq h\leq r+1$}\\
      0 & \text{ if $h>r+1$}.
     \end{dcases}
 \end{equation}
 In particular, as $t\to\infty$,
 \begin{equation}\label{eq:shapelimVar}
  \bVar\left[\frac{C_n^\shape(t)-\mu_n^\shape t}{\sqrt{t}}\right] \to \mb K_\shape(0) + 2 \sum_{h\geq 1} \mb K_\shape(h).
 \end{equation}
\end{lemma}

\begin{proof}
 We first note that \eqref{eq:shapeVar} is obtained by expanding the variance as 
 \begin{multline*}
  \bVar\left[C_n^\shape(t)\right]\\
  = \sum_{s=0}^{t-r-1} \bVar\left[Y_n^\shape(s)\right] + 2 \sum_{s=0}^{t-r-1} \sum_{h=1}^{t-r-1-s} \bCov\left[Y_n^\shape(s),Y_n^\shape(s+h)\right]. 
 \end{multline*}
 Grouping together the summands according to the value of $h$ and observing that the expressions given by Lemma \ref{lemma:EYn} are independent of $s$ gives \eqref{eq:shapeVar}. The limit in \eqref{eq:shapelimVar} is an immediate consequence, since $\mb K_\shape(h)=0$ for all $h>r+1$, so the sum in \eqref{eq:shapeVar} only contains non-vanishing terms for which $h=O(1)$. 
 
 It remains to verify \eqref{eq:autocov}. For $h=0$, we obtain by \eqref{eq:Ynkshapes1} for $\shape=\shape'$ that
 \begin{align*}
  &\bE\left[Y_n^\shape(s)^2\right]
  = \sum_{k=1}^n \bE\left[Y_{n,k}^\shape(s)\right] + 2 \sum_{k=1}^n \sum_{k'=1}^{k-1} \bE\left[Y_{n,k}^\shape(s) Y_{n,k'}^\shape(s)\right]\\
  &\qquad= \mu_n^\shape + 2 \frac{\prod_{i\geq 1} (n)_{4a_i}}{\prod_{j\geq 0} \fdf{2n-1}_{2b_j}} \sum_{k=1}^n (k-2a_0-1)_{2a_0-1} \sum_{k'=1}^{k-1} (k'-1)_{2a_0-1}.
 \end{align*}
 By \eqref{eq:disccalc}, the rightmost sum evaluates to $(2a_0)^{-1}(k-1)_{2a_0}$, which in turn simplifies the expression to 
 \begin{align*}
  \bE\left[Y_n^\shape(s)^2\right]
  &= \mu_n^\shape + \frac{\prod_{i\geq 1} (n)_{4a_i}}{a_0 \prod_{j\geq 0} \fdf{2n-1}_{2b_j}} \sum_{k=1}^n (k-1)_{4a_0-1}\\ 
  &= \mu_n^\shape + \frac{1}{4a_0^2}\prod_{i\geq0}\frac{(n)_{4a_i}}{\fdf{2n-1}_{2b_i}}
 \end{align*}
 after applying \eqref{eq:disccalc} a second time. The value for $\mb K_\shape(0)$ now follows. 
 
 The case $1\leq h\leq r+1$ is concluded by a similar argument starting from \eqref{eq:Ynkshapes2} with $a_i=a'_i$: After evaluating the sums over $k$ and $k'$, one arrives at 
 \[
  \bE\left[Y_n^\shape(s)Y_n^\shape(s+h)\right]
  = \frac{(n)_{2a_0}^2}{4a_0^2} \cdot \frac{(n-2a_0)_{2a_h}\prod_{h\neq i\geq 1} (n)_{2a_i+2a_{i-h}}}{\prod_{j\geq 0} \fdf{2n-1}_{b_j+b_{j-h}} }
 \]
 and \eqref{eq:autocov} then follows from simplifying the right-hand side using $a_{-h}=0$ and $(n)_{2a_0}(n-2a_0)_{2a_h}=(n)_{2a_h+2a_0}$. 
 
 Finally, if $h>r+1$ then the edges for the two components starting at levels $s$ and $s+h$ come from disjoint sets of layers, and hence are independent from one another. Thus, the covariance is $0$. 
\end{proof}

We now have all the tools to prove Theorem \ref{theorem:shape}.

\begin{proof}[Proof of Theorem~\ref{theorem:shape}]
 From \eqref{eq:Ynshape} we get for a strong shape $\shape$ with stretch $r\geq 0$, as $t\to\infty$, 
 \begin{equation}\label{eq:shapeLLN1}
  \frac{\bE\left[C_n^\shape(t)\right]}{t} = \frac{1}{t}\sum_{s=0}^{t-r-1} \bE\left[Y_n^\shape(s)\right]
  = \frac{1}{t}\sum_{s=0}^{t-r-1} \mu_n^\shape \to \mu_n^\shape.
 \end{equation}
 We thus established convergence in expectation for \eqref{eq:sshapeLLN}. 

 Almost sure convergence can be shown via the same renewal-theoretic argument as in the proof of Theorem~\ref{theorem:Soo}. Indeed, observe that if $K$ is a closed loop with stretch $r$ starting in level $s$, then none of the layers $s+1,\dots,s+r$ can be resets. In other words, closed loops are confined to the renewal intervals of the sling process. It follows that $\left\{C_n^\shape(t)\right\}_{t\geq0}$ is a process with regenerative increments over $T_j$, so almost sure convergence in \eqref{eq:sshapeLLN} is again obtained by invoking \cite[Theorem 2.54]{Serfozo09} and the central limit theorem follows from Theorem~\ref{theorem:regCLT} (i.e. \cite[Theorem~2.65]{Serfozo09}). As in the proof of Theorem~\ref{theorem:Soo}, applying Lemma~\ref{lemma:SooUI} yields the convergence of  moments, and, together with \eqref{eq:shapelimVar}, that 
 \begin{equation}\label{eq:sigmashape}
  \sigma_{n,\shape}^2 = \mb K_\shape(0) + 2 \sum_{h\geq 1} \mb K_\shape(h).
 \end{equation}

 The statements concerning weak shapes are obtained entirely analogously upon replacing in the Lemmas~\ref{lemma:EYn} and \ref{lemma:shapeVar} quantities decorated with $\shape$ by their analogues decorated with $\mb a$, such as $C_n^\shape(t)$ by $C_n^{\mb a}(t)$, or $Y_{n,k}^\shape(s)$ by $Y_{n,k}^{\mb a}(s)$, and so on. The additional factor of $\gamma_n^{\mb a}$ in \eqref{eq:wshapeLLN}, which accounts for the number of strong shapes with given weak shape $\mb a$, stems from Lemma~\ref{lemma:gammana}. For the variance in \eqref{eq:wshapeCLT}, we note that 
 \begin{equation}\label{eq:VarCant}
  \bVar\left[C_n^{\mb a}(t)\right] = \gamma_n^{\mb a} \bVar\left[C_n^\shape(t)\right] + (\gamma_n^{\mb a})_2\bCov\left[C_n^\shape(t),C_n^{\shape'}(t)\right]
 \end{equation}
 for any two $\shape\neq \shape'$ with weak shape $\mb a$, using Lemma~\ref{lemma:gammana} and the fact that the expressions in Lemmas~\ref{lemma:EYn} and \ref{lemma:shapeVar} only depend on the weak shape. Moreover, we also have for $\shape, \shape'$ two strong shapes with weak shape $\mb a$:
 \[
  \bE\left[Y_{n,k}^\shape(s)Y_{n,k'}^{\shape'}(s+h)\right] = \bE\left[Y_{n,k}^\shape(s)Y_{n,k'}^\shape(s+h)\right]
 \]
 except when $\shape\neq\shape'$, $k=k'$, and $h=0$ (since two components of different strong shapes cannot start at the same vertex). Hence for $\shape\neq\shape'$ we have $\bCov\left[Y_n^\shape(s),Y_n^{\shape'}(s+h)\right]=\mb K_{\shape}(h)$ for $h>0$ and $\bCov\left[Y_n^\shape(s),Y_n^{\shape'}(s)\right]=\mb K_{\shape}(0)-\mu_n^\shape$ for $h=0$. Continuing the computation from \eqref{eq:VarCant}, we thus obtain
 \begin{multline*}
  \bVar\left[C_n^{\mb a}(t)\right] = (\gamma_n^{\mb a})^2\left((t-r)\mb K_\shape(0) + 2 \sum_{h=1}^{t-r-1} (t-r-h) \mb K_\shape(h)\right)\\
  - (\gamma_n^{\mb a})_2 (t-r)\mu_n^\shape
 \end{multline*}
 which implies 
 \[
  \bVar\left[\frac{C_n^{\mb a}(t)}{\sqrt{t}}\right] \to (\gamma_n^{\mb a})^2\left(\mb K_\shape(0) + 2 \sum_{h\geq1} \mb K_\shape(h)\right) - (\gamma_n^{\mb a})_2 \mu_n^\shape
 \]
 as $t\to\infty$, in accordance with \eqref{eq:wshapeCLT}.
\end{proof}

\subsection*{Acknowledgements} The authors would like to thank Svante Janson and Stephan Wagner for helpful discussions, as well as Jan Steinebrunner for his advice on the introduction.

\bibliographystyle{alphaurl}
\bibliography{Brauer}

\end{document}